\documentclass[11pt,reqno]{amsart}
\usepackage{amsmath,amssymb,amsthm}
\usepackage{fullpage, mathrsfs, bbm,prettyref,xcolor,appendix,mathrsfs,esint}
\usepackage{hyperref}
\hypersetup{
    colorlinks,
    linkcolor={red!50!black},
    citecolor={green!50!black},
    urlcolor={blue!80!black}
}

\numberwithin{equation}{section}
\numberwithin{figure}{section}
\theoremstyle{plain}
\newtheorem{thm}{\protect\theoremname}[section]
\newtheorem*{thm*}{\protect\theoremname}
  \theoremstyle{plain}
  \newtheorem{lem}[thm]{\protect\lemmaname}

  \newtheorem{cor}[thm]{\protect\corname}
  \theoremstyle{definition}
 \newtheorem{defn}[thm]{\protect\definitionname}
  \newtheorem{conjecture}{Conjecture}

  \newtheorem{rem}{Remark}

  \providecommand{\definitionname}{Definition}
  \providecommand{\lemmaname}{Lemma}
  \providecommand{\theoremname}{Theorem}
  \providecommand{\corname}{Corollary}
  
  \newcommand{\dlower}{\underline{d}}
  \newcommand{\dupper}{\overline{D}}

\newcommand{\E}{\mathbb{E}}	
\newcommand{\cA}{\mathcal{A}}
\newcommand{\cB}{\mathcal{B}}

\newcommand{\cF}{\mathcal{F}}

\newcommand{\cP}{\mathcal{P}}

\newcommand{\cM}{\mathcal{M}}

\renewcommand{\P}{\mathbb{P}}
\newcommand{\R}{\mathbb{R}}

\newcommand{\eps}{\epsilon}

\newcommand{\abs}[1]{\left\lvert#1\right\rvert}
\newcommand{\norm}[1]{\| #1\|}
\newcommand{\gibbs}[1]{\langle #1 \rangle}

\newcommand{\Var}{\operatorname{Var}}
\newcommand{\supp}{\operatorname{supp}}

\renewcommand{\liminf}{\varliminf}
\renewcommand{\limsup}{\varlimsup}

\newcommand{\spin}{\varepsilon}

\newrefformat{cor}{Corollary \ref{#1}}
\newrefformat{app}{Appendix \ref{#1}}

\newcommand{\CUT}{{\sf CUT}}

\def\Mcut{{\sf MCUT}}

\def\ER{Erd\H{o}s-R\'enyi } 

\def\Rg #1{G^{\mbox{\tiny\rm Reg}}(N, #1)}

\def\malphacut{{\sf Mcut}_{\alpha}}

\def\de{{\rm d}}

\definecolor{subh}{RGB}{255,0,0}

\title[Unbalanced Cut and Generalized SK]
{On the unbalanced cut problem and the generalized Sherrington-Kirkpatrick model}

\author{Aukosh Jagannath}
\address[Aukosh Jagannath]{Department of Statistics and Actuarial Science and Department of Mathematics, University of Waterloo}
\email{a.jagannath@uwaterloo.ca}

\author{Subhabrata Sen}
\address[Subhabrata Sen]{Department of Statistics, Harvard University}
\email{subhabratasen@fas.harvard.edu}


\date{\today}

\begin{document}

\begin{abstract}
We establish a strict asymptotic inequality between a class 
of graph partition problems on the sparse \ER and random regular graph ensembles with the same
average degree. 
Along the way, we establish a variational representation for the ground state energy
for generalized mixed $p$-spin glasses and derive strict comparison inequalities
for such models as the alphabet changes.
\end{abstract}

\maketitle

\section{Introduction}
Consider the following graph partition problem, called the \emph{unbalanced cut problem}.
Let $G=(V,E)$ be a graph and let $V_{1},V_{2}\subset V$
be a partition of the vertex set, $V$, such that 
\begin{align*}
\abs{V_{1}}  =\alpha\abs{V} \qquad \text{ and } \qquad
\abs{V_{2}}  =(1-\alpha)\abs{V},
\end{align*}
where $0<\alpha<1/2$ is a fixed number.
Let $\CUT(V_{1},V_{2})$ denote the number of edges joining $V_1$ to $V_2$. 
We are interested in the maximum
of this quantity over all such partitions, which we denote by 
\begin{align}
\Mcut_{\alpha}(G)=\max_{\substack{\abs{V_{1}}=\alpha\abs{V}\\
\abs{V_{2}}=(1-\alpha)\abs{V}
}
}\CUT(V_{1},V_{2}). \label{eq:Mcut_def}
\end{align}
Observe, for example, that when $\alpha=1/2$ this is the maximum bisection problem (Note that the quantity is only well-defined when $\alpha$ is rational).

We aim to compare this quantity asymptotically between two well-known 
random graph models. The first ensemble we consider is
the sparse \ER random graph, $G(N,\frac{d}{N})$, where each
edge is added independently with probability $d/N$, where $d$ is
a fixed constant. The second ensemble we
consider is the random $d$-regular graph, $\Rg{d}$, where 
a $d$-regular graph on $N$ vertices is selected uniformly at random. 
These random graphs are typically sparse: such a graph on $N$ vertices has $O(N)$ edges with high probability.
Our main result is a strict comparison between the unbalanced cut problem on these two ensembles in \emph{the mean field approximation}, i.e., for asymptotically large degrees. 
\begin{thm}
\label{thm:cut_difference}
For any $0< \alpha < \frac{1}{2}$, 
there is a constant $C(\alpha)>0$ such that
\begin{align}
\liminf_{d\to\infty}\liminf_{N\to\infty}\frac{\Mcut_{\alpha}(G(N,\frac{d}{N}))-\Mcut_{\alpha}(\Rg{d})}{\sqrt{d}N}\geq C(\alpha) \label{eq:cut_difference_ineq}
\end{align}
almost surely.
\end{thm}

The  novelty in this inequality is the fact that $C(\alpha)$ is strictly positive. Indeed, for the maximum bisection problem ($\alpha=1/2$) it is known that that this difference is in fact zero \cite{dembo2015extremal}.
From a combinatorial perspective, this inequality is surprising: intuitively, it suggests
that the rigidity of the edge structure of random regular graphs, in comparison to that of \ER graphs, has macroscopic ramifications for cut problems. 
The curious reader might also wonder if $C(\alpha)$ is 
in fact the sharp constant. We do not believe that this is the case. Instead, our
approach yields a natural conjecture regarding the sharp constant. 
in terms of a minimizer of $\cP_T$. We discuss this conjecture in Section \ref{sec:conj1}.   

The core of our approach is a connection between
the unbalanced cut problem --- \emph{a priori} a question of pure combinatorics--- to
the ground state energy of the \emph{Generalized Sherrington-Kirkpatrick} model---\emph{a priori} a question
of statistical physics.  Before presenting this connection,
let us first place it in context.

\begin{rem}
If instead of $\Mcut_{\alpha}(G)$ one considers the minimum $\alpha$-cut problem, i.e., taking a minimum in \eqref{eq:Mcut_def} instead of a maximum, the inequality \eqref{eq:cut_difference_ineq} is reversed in the obvious way with no change to the proof.
\end{rem}

\subsection{Background}
\label{subsection:background}
Graph partition problems are classical combinatorial optimization problems, having applications in Computer Science, 
Statistics and Machine Learning \cite{newman2013community, DiazReview,buluc2016partitioning}. These problems are described as follows. 
Given a graph $G= (V,E)$, we seek to divide the set of vertices into two or more parts 
such that the number of edges between the distinct parts is optimized.  
For example, the well-known MaxCut problem seeks to partition the vertex set into two parts, 
$V_1$ and $V_2$, such that the number of edges connecting the two parts, $\CUT(V_1,V_2)$, is maximized. 
Another example  is the maximum bisection problem,  $\Mcut_{1/2}(G)$.
The study of these problems in the
sparse regime has received much attention
from the combinatorics community \cite{Alon97,Bollobas84,Diaz07}, though
they remain very challenging.
Graph partition problems have also been studied extensively in the physics literature
 as they are predicted to lie in a canonical class of models called \emph{spin glasses}. 

Connections between spin glasses
and combinatorial optimization problems are by now classical observations. 
This perspective has received a tremendous amount of attention
in the physics, mathematics, and combinatorics literatures. It is impossible
to provide here anywhere near a complete survey of this literature. Instead we point the reader to the texts
\cite{mezard2009information, MPV}. Although many of the predictions of the physics literature
have been verified, we are very far from understanding the full picture.
For a sample of recent, rigorous results in this direction see \cite{aldous2001assignment,dss_3,nair2005assignment, steele1997optimization}.

In the setting of graph partition problems, this connection goes back at least to the work of Fu--Anderson \cite{fu1986application}. 
 Recently, there has been significant progress in formalizing this connection. 
First, Bayati--Gamarnik--Tetali \cite{BGT}  
explored this connection by using a sub-additivity argument to establish the existence of a deterministic 
limit for the MaxCut on sparse \ER and random regular graphs. In this light, 
it is natural to study this deterministic limiting value as a function of the degree, $d$.

In the large degree limit, or mean field approximation, {it is not hard to see that the leading order contribution is of order $d$ and is essentially  the expected value of the objective function.
Evidently, the heart of the matter is then in subsequent terms of the expansion in $d$. Indeed, the next term, often of order $\sqrt{d}$, is highly nontrivial, and related to the ground state energy of mean field spin glasses.}
This idea was partially formalized by Dembo, Montanari, and one of the authors in  \cite{dembo2015extremal}, 
where it was shown that asymptotically
first in the vertex number and then in the degree, the normalized
MaxCut, maximum bisection, and the minimum bisection of the \ER and 
the random $d$-regular graph ensembles are equal to second order in $d$.
Again, the first order contribution is that of a random labeling 
of the vertices, $d/4$. The second order term, of order $\sqrt{d}$, is (essentially)
the ground state energy of the Sherrington-Kirkpatrick model \cite{sherrington1975solvable}.

In subsequent work, one of the authors \cite{sen2016optimization} generalized this result 
to a family of combinatorial optimization problems, where the objective is of a tensorial nature. 
For a general class of these problems on \ER or random regular hypergraphs, 
a similar asymptotic appears, where this time the order $\sqrt{d}$ term is the ground 
state energy of a suitably chosen spin glass model. {
From this perspective, it is natural to believe that for a wide class of these problems,
the optimal value should be the same on these two ensembles, up to $o(\sqrt{d})$ corrections.
To this end, \cite{sen2016optimization} derives broad sufficient conditions for the normalized maxima to have the same value, up to lower order contributions in $d$.  }
{In this paper, we show, surprisingly, that this belief is flawed. 
Theorem \ref{thm:cut_difference} establishes that even a small perturbation of the maximum bisection problem has widely different behavior on \ER and random regular graphs. }

Theorem \ref{thm:cut_difference} is also significant for a number of conceptual reasons. First, it establishes a strict inequality between these statistics on \ER and random regular graphs--- which is difficult to establish using purely combinatorial techniques. Second,
 aside
from solving an interesting question of combinatorics, it leads us to resolve an important  question
of  independent interest in the theory of spin glasses, namely the ground state energy of the generalized mixed $p$-spin glass model. {Resolving these spin-glass questions is in fact our main technical contribution in this paper.}

\subsection{Generalized mixed $p$-spin models and their connection to Theorem \ref{thm:cut_difference}}
At the heart of \prettyref{thm:cut_difference} is a connection between the unbalanced cut problem and what are called   \emph{Generalized mixed $p$-spin models} which were introduced by Panchenko in \cite{panchenko2005generalized}. These are natural
generalizations of the Ising $p$-spin model \cite{derrida1981random} to the case where the spins take values in a  finite alphabet 
$\Sigma\subset\R$. (The Ising $p$-spin model corresponds to $\Sigma=\{\pm1\}$.)

More precisely, let $\Sigma\subset\R$ be
a finite set called \emph{the alphabet}, and let the \emph{configuration space} be defined as
$\Sigma^{N}$. The \emph{Hamiltonian} for this model is the centered
Gaussian process indexed by $\Sigma^{N}$ with covariance 
\begin{equation}\label{eq:ham-def}
\E H_{N}(\sigma^{1})H_{N}(\sigma^{2})=N\xi(R(\sigma^{1},\sigma^{2})),
\end{equation}
where $\xi(t)=\sum_{p\geq2}\beta_{2p}^{2}t^{2p}$ is an even power
series and 
\[
R(\sigma^{1},\sigma^{2})=\frac{1}{N}\sum_{i=1}^{N}\sigma_{i}^{1}\sigma_{i}^{2},
\]
is called the \emph{overlap}. Let 
\begin{align*}
\dupper  =\max_{\epsilon\in\Sigma}\epsilon^{2} \qquad\text{ and }\qquad
\dlower  =\min_{\epsilon\in\Sigma}\epsilon^{2}.
\end{align*}
We assume that $\xi(\dupper+\epsilon)<\infty$ for some $\eps>0$ so
that this process is well-defined. 

The application to graph-partition problems (Theorem \ref{thm:cut_difference}) motivates our interest  in the restricted
normalized ground state energy of this process, 
\[
GS_{N}(A_{N}):=\frac{\max_{\sigma\in A_{N}}H_{N}(\sigma)}{N},
\]
where $A_{N}\subset\Sigma^{N}$. Specifically, we are interested in
two cases, either $A_{N}=\Sigma^{N}$ or 
\begin{equation}\label{eq:A_N-def}
A_{N}=A_{N}(T,\epsilon_{N})=\left\{ \sigma\in\Sigma^{N}:R(\sigma,\sigma)\in(T-\epsilon_{N},T+\epsilon_{N})\right\} ,
\end{equation}
for some $\epsilon_{N}\to0$ sufficiently slowly.

To understand why, we will show by an application of the results of \cite{sen2016optimization} that
the proof of Theorem \ref{thm:cut_difference} can be reduced to a strict comparison between 
the limiting ground state energies of two generalized Sherrington-Kirkpatrick (SK) models--- models for which $\xi(t)=2 t^2$.
For a formal statement, we refer the reader to Lemma \ref{lemma:gaussiancomp}. For the \ER graph, we obtain 
the SK model, while for the random regular graph, we obtain a generalized SK model, 
both constrained in a certain natural fashion. 
Consequently, if an inequality sufficed, one could then use a well-known
 Guerra-type \cite{guerra2003broken} or Slepian-type \cite{LedouxTalagrand} interpolation
to easily obtain the desired estimate. 
We are interested, however,
in a \emph{strict} inequality asymptotically in $N$.

To accomplish this goal, our approach is to provide a quantitative understanding 
of the derivative of this interpolation. 
This is accomplished by a fine analysis of the limiting ground state energy in generalized mixed $p$-spin models. 
The derivative is naturally related to the minimizers
of a certain family of variational problems called ``Parisi variational problems'',
and the question pertains to the scaling of the minimizers of these problem as one 
tunes a certain parameter, called the temperature. This naturally
leads us to a question of $\Gamma$-convergence of such variational problems \cite{Braides02,DalMaso93}. 
This approach will not only yield that this constant is positive but will also yield 
a natural conjecture as to the sharp constant. 
We explain this in greater detail in  \prettyref{sec:annealing-intro}
 and 
discuss the aforementioned conjecture in  \prettyref{sec:conj1}.

 \subsection{Ground State Energies of Generalized Mixed $p$-spin models}
In this section, we explain our results regarding
variational representations for ground state energies of generalized mixed $p$-spin models.

The question of the ground state energy is natural from a statistical physics perspective, and has recently received considerable attention in the
mathematics literature. In the case that the configuration space is the sphere, $S^{N-1}(\sqrt{N})$, a variational formula
was independently provided by Chen--Sen\cite{chensen2017} and Tobasco with one of the authors \cite{jagTob17}.
In the case that the alphabet is $\Sigma = \{\pm 1\}$, called the Ising spin setting, a variational representation was obtained by Auffinger--Chen in \cite{auffchen2017}. These representations have since been 
used to study a wide variety of questions  \cite{auffinger2017energy,ACZ17, chen2017local,  chl2016energylandscape, chenPanch2017disorder} . 

We derive a variational representation for the normalized ground state energy of generalized mixed $p$-spin models as a consequence of our approach. 
To this end, we introduce the following notation. 

Let $T\in[\dlower,\dupper]$. Let $\cM([0,T])$ be
the positive cone of Radon measures on $[0,T]$.
Let $\cA_{T}\subset\cM([0,T])$ be the set of measures of the
form 
\[
\cA_{T}=\left\{ \nu\in\cM([0,T]):\nu=m(t)dt+c\delta_{T},\quad m(t)\geq0\text{ non-decreasing and cadlag}\right\},
\]
equipped with the weak-{*} topology. {Note, in particular, finitness enforces $\int_0^{T} m(t) dt < \infty$}.  On this space we define the
ground state energy functional $\cP_T:\cA_{T}\times\R\to\R$. For $\nu=mdt+c\delta_{T}$,
we let
\begin{equation}\label{eq:local-zero-temp-pfunc}
\cP_T(\nu,\lambda)=u_{\nu,\lambda}(0,0)-\lambda T-\frac{1}{2}\int_{0}^{T}\xi''(s)sd\nu(s),
\end{equation}
where $u_{\nu,\lambda}$ is the unique weak solution to
\begin{equation}\label{eq:zero-temp-pde}
\begin{cases}
\partial_{t}u+\frac{\xi''}{2}(\Delta u+m(s)\left(\partial_{x}u\right)^{2})=0 & (t,x)\in[0,T) \times\R\\
u(T,x)=f(x,\lambda,c) & x\in \R,
\end{cases}
\end{equation}
where 
\begin{equation}\label{eq:zero-temp-bd}
f(x,\lambda,c)= \begin{cases}
\sup_{\epsilon\in\Sigma}\left\{\epsilon x+\left(\lambda+\frac{\xi''(T)}{2}\cdot c\right)\epsilon^{2}\right\} & \mathrm{if }\,\, T \in (\dlower, \dupper), \\
 \max_{\epsilon^2= T} \left\{\epsilon x+\left(\lambda+\frac{\xi''(T)}{2}\cdot c\right)\epsilon^{2}\right\} & \mathrm{if }\,\, T \in \{ \dlower, \dupper\}. 
\end{cases}
\end{equation}
(For a notion of weak solution of such PDEs see \cite{jagannath2015dynamic}
and for  basic regularity in this setting see \prettyref{app:pde}.) We then have
the following. 
\begin{thm}
\label{thm:variational_rep}
For any $\epsilon_{N}\to0$ sufficiently slowly, 
\begin{equation}\label{eq:GS-vp-constrained}
\lim_{N\to\infty}GS_{N}(A_{N})=\inf_{\substack{\nu\in\cA_{T}\\
\lambda\in\R
}
}\cP_T(\nu,\lambda)
\end{equation}
almost surely. Furthermore 
\[
\lim_{N\to\infty}GS_{N}(\Sigma^{N})=\sup_{T}\inf_{\substack{\nu\in\cA_{T}\\
\lambda\in\R
}
}\cP_T(\nu,\lambda)
\]
almost surely. 
\end{thm}

\begin{rem}

We note here that one can eliminate
the dependence of this variational problem on $c$ by making the substitution $\lambda\mapsto \lambda- \frac{c}{2}\xi''(T)$.
We leave the problem in this form for two key reasons: first, our derivation of this result will be by way of $\Gamma$-convergence for which one must allow $c$ to be non-zero (see the discussion in the next section); second,
our main application, the proof of \prettyref{thm:cut_difference}, will use this formula and said $\Gamma$-convergence result to characterize limit points of certain sequences of measures, which may have non-zero $c$. For the discussion of the physical interpretation 
of $c$ see \cite{chensen2017,jagTob17} and \prettyref{app:appendix-at}.
\end{rem}

To return to our combinatorial motivations, 
let us begin by first observing that as a corollary 
to \prettyref{thm:variational_rep} and en route to proving \prettyref{thm:cut_difference}, we also provide explicit formulas for $\Mcut_\alpha$ to second order
in the degree. In the following, we let $\cP^1_T(\nu,\lambda)$ denote the functional 
\prettyref{eq:local-zero-temp-pfunc} with $\Sigma=\{\pm 1-(2\alpha -1)\}$ and $\xi(t)=2t^2$. Let  $$T(\alpha)=4\alpha(1-\alpha).$$ Finally, let  $\cP^2:\cA_1\times\R\to\R$ denote the functional
\begin{align}
\mathcal{P}^{2}(\nu,h)=u_{\nu,0}(0,h)-\frac{1}{2}\int_{0}^{1}\xi''(s)sd\nu(s), \label{defn:P2}
\end{align}
where $u$ is the unique solution to \eqref{eq:zero-temp-pde} with alphabet $\Sigma =\{+1,-1\}$ and $\xi(t)=2t^2$, i.e.,
with final time data $f(x,0,c)=|x|+2c$. 

\begin{cor}\label{cor:cut-rep}
For any $0<\alpha<1/2$ we have that,
\begin{align}
\lim_{d\to\infty}\lim_{N\to\infty} \frac{\Mcut_{\alpha}(\Rg{d})-N d\alpha(1-\alpha)}{\sqrt{d}N} &= \frac{1}{4}  \inf_{\nu,\lambda}\cP_{T(\alpha)}^1(\nu,\lambda)\\
\lim_{d\to\infty}\lim_{N\to\infty} \frac{\Mcut_{\alpha}(G(N,\frac{d}{N}))-N d\alpha(1-\alpha)}{\sqrt{d}N} &= \frac{1}{4} \inf_{\nu,h}[ \cP^2(\nu,h)-(2\alpha-1)h].
\end{align}
\end{cor}

\subsection{An analytical approach to annealing}\label{sec:annealing-intro}
At the heart of the recent work regarding variational representations for ground state energies
is an analytical approach to the notion of \emph{annealing}. 
Annealing,
that is, adding a temperature and sending it to zero, is natural from 
the point of view of statistical physics
and underlies well-known algorithms for optimization \cite{kirkpatrick1983optimization}.

The idea, roughly, is as follows.
The ground state energy can be computed as the limit of an important quantity called the \emph{free energy} 
which is defined as follows. Recall the Hamiltonian $H_N$ from \eqref{eq:ham-def}. 
The \emph{free energy} at inverse temperature $\beta$ is defined as 
\[
F_N(\beta,\xi) = \frac{1}{N}\log \int_{\Sigma^N} e^{\beta H_N(\sigma)}d\sigma,
\]
where $d\sigma$ is the counting measure, and the \emph{restricted} free energy corresponding to a set $A\subset\Sigma^N$ 
and inverse temperature $\beta$ is defined as
\[
F_N(\beta,\xi;A) = \frac{1}{N}\log \int_{A} e^{\beta H_N(\sigma)}d\sigma.
\]
In our setting, we are interested in $A_N$ of the form \prettyref{eq:A_N-def}.

In these models, a variational expression for the free energy at a fixed temperature
is obtained using a  ``Parisi-type formula''. 
For  $\Sigma=\{\pm 1\}$, this was proved by Talagrand in \cite{TalagrandFormula}
and Panchenko in \cite{panchenko2014mixed}. 
For general alphabets, the variational problem was derived by Panchenko in \cite{panchenko2005generalized} 
(and more recently again in \cite{panchenko2015free}) 
where he showed that for any $T\in [d,D]$ and any $\eps_N\to0$ sufficiently slowly
\begin{align}
F_N(\beta,\xi;A_N(T,\eps_N))&\to F(\beta,\xi;T)\label{eq:loc-fe-conv}\\
F_N(\beta,\xi)&\to F(\beta,\xi)
\end{align}
where 
\begin{align}
F(\beta,\xi;T)&= \beta\inf_{\nu\in X_{\beta, T},\lambda\in\R}\cP_{\beta,T}(\nu,\lambda)\label{eq:loc-fe}\\
F(\beta,\xi)&\to \beta\sup_{T\in[\dlower,\dupper]} F(\beta,\xi;T)
\end{align}
Here $\cP_{\beta,T}$ is called the local Parisi functional. For its precise definition see \prettyref{eq:pfunc-def}.
It is not difficult to see that 
\[
\frac{1}{\beta}F(\beta,\xi;T)\stackrel{\beta\to\infty}{\longrightarrow}\lim_{N\to\infty} GS_N(A_N).
\]
Thus the question of ground state energies is related to the large $\beta$ limit
of these variational problems. 

A natural approach to the asymptotic analysis of variational problems is
De Georgi's notion of $\Gamma$-convergence \cite{Braides02,DalMaso93}.
Our approach, following \cite{jagTob17}, is to study the $\Gamma$-limit of $\cP_{\beta,T}$. As a direct consequence, we obtain a variational representation for the limiting ground state energy, similar to \cite{auffchen2017,chensen2017,jagTob17}. Further, this allows us to control zero temperature asymptotics of physically relevant quantities, and derive strict comparison inequalities. 
We remark here that
upper bounds only require the $\Gamma$-liminf inequality, and have been used in the recent progress in  
\cite{auffinger2017energy, chl2016energylandscape,chenPanch2017disorder}.

Due to the natural topology of the $\Gamma$-limit,
one formally expects the need to understand how the nonlinear term ---
the solution in space-time of a Hamilton--Jacobi--Bellman
equation, where
the coefficient of the non-linearity is the variable of optimization ---
behaves as one allows this coefficient to become the derivative
of a Dirac mass. More precisely, one needs an appropriate
limiting notion of solution for such situations. (This explanation is necessarily vague,
for a more precise description see \prettyref{sec:gamma-parisi-intro}.)
Auffinger--Chen \cite{auffchen2017} observed that in the case $\Sigma =\{\pm1\}$,
the linear term in the functional exactly cancels this effect, allowing one 
to avoid this issue.
If one perturbs the problem
by allowing the spins to take values $\{\pm1+\epsilon\}$, however,
the arguments
in the literature do not apply. We are then forced to tackle the question
of the limit of the non-linear term.
To this end, we introduce a notion of \emph{annealed solution} which yields an interpretation for the solution of the
PDE in this singular regime as an appropriate zero-temperature $\Gamma$-limit.
The $\Gamma$-convergence of $\cP_{\beta,T}$ then follows.

\subsection*{Acknowledgements}
The authors thank Amir Dembo for introducing them to the unbalanced cut problem.
The authors thank Jonathan Shi for pointing out an error in an earlier version of this manuscript.  
A.J. thanks Ian Tobasco for fruitful discussions,
 as well as the University of Toronto and Harvard University
mathematics departments for their hospitality where part of this research was conducted. 
This research was conducted while A.J. was supported by NSF OISE-1604232
and NSERC [RGPIN-2020-04597, DGECR-2020-00199].
Cette recherche a \'et\'e financ\'ee par le Conseil de recherches en sciences naturelles et en g\'enie du Canada (CRSNG).

\subsection{Outline of proof of Theorem 1.1}

Let us now briefly outline the proof of Theorem 1.1. The starting point in our argument is Lemma \prettyref{lemma:gaussiancomp}, which establishes that one can approximate 
$\Mcut_{\alpha}$ on Erd\H{o}s-R\'{e}nyi and random regular graphs up to $o(\sqrt{d})$ corrections 
by the ground state energies of certain Generalized SK models. 
We then apply a Gaussian interpolation argument in \prettyref{lem:int_by_parts} to compare the corresponding free energies. 
The error term in this comparison is a quantity that depends on the minimizer of the corresponding variational problems. Finally, we send $\beta\to\infty$, and analyze the limiting variational problem to prove a sign on the limit of this error term in Theorem \ref{thm:gamma_conv_application}.

\subsection{Outline of paper}
The remainder of this paper is organized as follows.
In the next section, we reduce the proof of \prettyref{thm:cut_difference}
to an inequality about asymptotics of ``overlaps'' of spin glass models (we introduce
this notion presently). 
In order to study this question, we introduce, in \prettyref{sec:gamma-parisi-intro},
 the Parisi boundary value
problem and the notion of annealed solutions to this problem. 
In \prettyref{sec:loc-free-energy}, we present 
Panchenko's Parisi-type formula for the Free energy in this setting
and use the notion of annealed solutions to compute its $\Gamma$-limit. 
We then prove pre-compactness and convergence of the minimizers of this problem. 
We then turn briefly in \prettyref{sec:analysis-zero-temp} to computing the first variation of the ground 
state functional $\cP_T$. Finally, Section \ref{subsec:gamma_conv_app} establishes the main spin glass estimate necessary for \prettyref{thm:cut_difference}.
For the benefit of the reader, we briefly present some basic analytical and topological results
used in this paper in the appendix.

\section{The Unbalanced Cut Problem}\label{sec:unbalanced-cut}
In this section, we establish Theorem \ref{thm:cut_difference}. To this end, 
let $G_1 \sim G(N, \frac{d}{N})$ and $G_2 \sim \Rg d$ be defined on the same probability space $(\mathbf{\Omega}, \mathscr{F}, \P)$. Recall that the random graph $G(N, \frac{d}{N})$ has vertex set $[N]$, and the edges are added independently with probability $\frac{d}{N}$ each. For ease of computation, we consider $\Rg d$ to be drawn from the configuration model \cite{bollobas_configuration}. While this is a multi-graph in general, it is easy to see that conditioned on simplicity, the graph obtained is actually uniformly distributed. Further, the probability of the obtained graph being simple is bounded away from zero (see for example \cite{wormald} and references therein). Thus it suffices to establish our result for the configuration model.

\subsection{Concentration of $\malphacut$}
Our first result establishes the concentration properties of $\malphacut$ on $G_1$ and $G_2$ around their respective expectations. 

\begin{lem}
\label{lem:concentration}
For any $\varepsilon>0$ sufficiently small, there exists a universal constant $C(d, \epsilon)>0$ 
\begin{align}
\P \Big[ \Big| \frac{\malphacut(G_1)}{N} - \E\Big[ \frac{\malphacut(G_1)}{N}\Big] \Big| > \varepsilon \Big] &\leq 5 \exp[- C(d, \varepsilon) N  \Big] \nonumber \\
\P \Big[ \Big| \frac{\malphacut(G_2)}{N} - \E\Big[ \frac{\malphacut(G_2)}{N}\Big] \Big| > \varepsilon \Big] &\leq  2 \exp\Big[- \frac{N \varepsilon^2}{d} \Big]{.} \nonumber
\end{align}
\end{lem}

\begin{proof}
The concentration argument for  random regular graph $G_2$ follows immediately upon an application of \cite[Theorem 2.19]{wormald}. 
Next, we establish the result for the \ER random graph $G_1$. To this end, 
let $\abs{E}$ denote the number of edges in $G_1$ and observe that  
$|E| \sim {\textrm{Bin}}\Big( {N \choose 2}, \frac{d}{N}  \Big)$. Thus $\E[|E|] = \frac{(N-1)d}{2}$. We have, 
\begin{align}
&\P\Big[ \Big| \malphacut(G_1) - \E[\malphacut(G_1) ] \Big| > N \varepsilon \Big] \nonumber \\
&\leq \P\Big[ \Big| \malphacut(G_1) - \E\Big[\malphacut(G_1) \Big| |E| \Big]  \Big| > \frac{N\varepsilon}{2} \Big] + \P\Big[ \Big| \E\Big[\malphacut(G_1) \Big| |E| \Big]  - \E\Big[\malphacut(G_1) \Big] \Big| > \frac{N \varepsilon}{2} \Big]. \nonumber  \\
& := I + II. \nonumber
\end{align}

To control $I$, we proceed as follows. We set $\mathcal{E} = \{  |E| \leq  \E[|E|] + N \varepsilon_0 \}$, for some $\varepsilon_0>0$ to be chosen appropriately. 
\begin{align}
I &\leq \E\Big[ \mathbf{1}(\mathcal{E}) \P \Big[ \Big| \malphacut(G_1) - \E\Big[ \malphacut(G_1) \Big| |E| \Big] \Big| > \frac{N \varepsilon}{2} \Big| |E| \Big]  + \P [ \mathcal{E}^c]. \nonumber\\
&\leq 2 \exp\Big[ - \frac{N^2 \varepsilon^2}{8 \Big[ \E[|E|] + N \varepsilon_0 \Big]} \Big] + \exp\Big[ - \frac{2}{3} \frac{N \varepsilon_0^2}{d} \Big], \nonumber
\end{align}
where we bound the first term using the Azuma-Hoeffding inequality on the traditional edge-exposure martingale, and the second term by the Chernoff bound. 

To control $II$, let $G_1' = ([N], E') \sim G(N, \frac{d}{N})$ be an \ER random graph independent of $G_1$. We claim that 
\begin{align}
\Big| \E\Big[\malphacut(G_1) \Big| |E| \Big] - \E[ \malphacut(G_1') \Big| |E'| \Big]  \Big| \leq \Big| |E| - |E'| \Big|. \label{eq:claim}
\end{align}
Given the claim, we have, using Jensen's inequality, 
\begin{align}
&\Big| \E\Big[ \malphacut(G_1) \Big| |E|  \Big] - \E\Big[\malphacut(G_1) \Big] \Big| \leq \E_{|E'|} \Big|  \E\Big[ \malphacut(G_1) \Big| |E|  \Big]  -  \E\Big[ \malphacut(G_1') \Big| |E'|  \Big]  \Big| \nonumber \\
&\leq \E_{|E'|} \Big[ \Big| |E| - |E'| \Big| \Big] \leq \Big| |E| - \E[|E|] \Big| + \E\Big| |E'| - \E\Big[ |E'| \Big] \Big| \leq  \Big| |E| - \E[|E|] \Big| + C \sqrt{N d},  \nonumber
\end{align}
for some constant $C>0$, where the last inequality follows using Cauchy-Schwarz. Thus we have the bound,
\begin{align}
II < \P \Big[ \Big| |E| - \E[|E|]  \Big| > \frac{N \varepsilon}{2} - C \sqrt{Nd} \Big] \leq 2 \exp\Big[-\frac{ N \varepsilon^2}{6d} \Big],\nonumber 
\end{align}
where the last inequality follows using the Chernoff bound. 
 This completes the proof{,} modulo the claim \eqref{eq:claim}, once we optimize over $\varepsilon_0$. To prove this claim, we proceed as follows.

Given $|E|, |E'|$, we will construct a coupling $(H, H')$ such that marginally, $H$ and $H'$ are distributed as $G_1$ conditioned to have $|E|$ and  $|E'|$ edges respectively. Assume without loss of generality, that $|E'| > |E|$. Start with an empty graph on $[N]$. Add edges sequentially, uniformly at random. At the end of $|E|$ steps, 
call the graph formed $H$.  Continue adding edges, and at the end of $|E'|$ steps call the graph $H'$. Under this construction, 
\begin{align}
\Big| \malphacut(H) - \malphacut(H') \Big| \leq \Big| |E'| - |E| \Big| \nonumber
\end{align} 
almost surely. Taking the expectation of this inequality with respect to the joint law of $(H, H')$ and applying Jensen's inequality yields \eqref{eq:claim}, as desired. \end{proof}

\subsection{Comparison to a Gaussian problem} 
In light of Lemma \ref{lem:concentration}, it suffices to compare the expectations of $\malphacut$ on $G_1$ and $G_2$. To this end, we introduce the following notation. 
For any graph $G= (V,E)$ with $|V| = N$, assume that  $V= [N]$ without loss of generality. Observe that every partition $V= V_1 \sqcup V_2$ of a graph can be represented by a vector $\sigma \in \{ \pm 1\}^N$, with the two parts being encoded as $V_1 = \{ i : \sigma_i =1\}$ and vice versa. Therefore, every partition with $|V_1| = \alpha N$, $|V_2| = (1-\alpha)N$ corresponds to a unique vector $\sigma \in \{\pm 1\}^N$ such that $\sum_i \sigma_i = N (2\alpha -1)$. We set 
 \begin{align}
 S_N(\alpha) = \Big\{ \sigma \in \{ \pm 1 \}^N : \sum_i \sigma_i = N (2 \alpha -1) \Big\}. \nonumber 
 \end{align}
 Next, we consider a GOE matrix $J = (J_{ij})_{N \times N}$ and define $g_i = \sum_j \frac{J_{ij}}{\sqrt{N}}$. For $\sigma \in S_N(\alpha)$, we define,
 \begin{align}
 H_0(\sigma) &= \sum_{ij} \frac{J_{ij}}{\sqrt{N}} \sigma_i \sigma_j,  \\
 H_1(\sigma) &= \sum_{ij} \frac{J_{ij}}{\sqrt{N}} \sigma_i \sigma_j - 2 (2\alpha -1) \sum_i g_i \sigma_i + (2\alpha-1)^2\sum_{ij}\frac{J_{ij}}{\sqrt{N}}. \label{eq:h1-def} 
 \end{align}
 We have the following lemma. 
 \begin{lem}
 \label{lemma:gaussiancomp}
 As $N\to \infty$, we have,
 \begin{align}
 \E\Big[ \frac{\malphacut(G_1)}{N} \Big] &= d \alpha (1-\alpha) + \frac{\sqrt{d}}{4N}\E\Big[ \max_{\sigma\in S_N(\alpha)} H_0(\sigma) \Big] + o(\sqrt{d}). \label{eq:er} \\
\E\Big[ \frac{\malphacut(G_2)}{N} \Big]&= d \alpha(1-\alpha) + \frac{\sqrt{d}}{4N} \E\Big[ \max_{\sigma \in S_N(\alpha)} H_1(\sigma)\Big] + o(\sqrt{d}). \label{eq:regular}
  \end{align}
 \end{lem}

 \begin{proof}
 We start with the proof of \eqref{eq:er}. This follows directly from \cite[Theorem 1.1]{sen2016optimization}. In this case, we have
 $p=2$, $A_N = S_N(\alpha)$, $f : \{-1, 1\}^2 \to \mathbb{R}$, $f(x,y)= \mathbf{1}_{\{x \neq y\}} = (1- x y)/2$ and $\kappa_1 = 1$. Further, we note that on the set $A_N$, the contribution from the expectation is exactly $\alpha (1- \alpha)$, and this completes the proof. 
 
 Next we consider \eqref{eq:regular}. This will be established using \cite[Theorem 1.2]{sen2016optimization}. Using the same setup as above, we obtain that for random regular graphs, 
 \begin{align}
\E\Big[ \frac{\malphacut(\Rg{d})}{N} \Big] = d \alpha (1-\alpha) + \frac{\sqrt{d}}{2N} \E\Big[ \max_{\sigma \in S_N(\alpha)}\Big[ \sum_{ij} \frac{J_{ij}}{\sqrt{N}} \mathbf{1}_{\{ \sigma_i  \neq \sigma_j \}} - \frac{2}{N} \sum_{ij} g_i \mathbf{1}_{\{ \sigma_i \neq \sigma_j \}} \Big] \Big] + o_{d}(\sqrt{d}). \nonumber 
 \end{align}
 We note that for $\sigma \in S_N(\alpha)$, $\sum_i \sigma_i = N(2 \alpha -1)$ and $\mathbf{1}_{\{x \neq y\}} = (1- xy)/2$. Plugging these into the equation above completes the proof upon noting that the last term in \eqref{eq:h1-def} is $o_N(1)$ with high probability. 
 \end{proof}

\subsection{Proof of Theorem \ref{thm:cut_difference}}
The following Theorem establishes a strict lower bound on the difference between limiting ground state energies of the two generalized SK models. The proof is deferred to section \ref{subsection:proof_diff}. 
\begin{thm}
\label{thm:diff_groundstate}
For $0< \alpha < \frac{1}{2}$, there exists a constant $C_0(\alpha)>0$ such that 
\begin{align}
\liminf_{N \to \infty} \frac{1}{N} \Big[ \E\Big[\max_{\sigma \in S_{N}(\alpha)} H_0(\sigma) \Big] - \E\Big[ \max_{\sigma \in S_N(\alpha)} H_1(\sigma) \Big]  \Big] > C_0(\alpha). \nonumber
\end{align} 
\end{thm}

 Now, we note that Lemma \ref{lemma:gaussiancomp} immediately implies that 
\begin{align}
\frac{\E[\malphacut(G_1)] - \E[\malphacut(G_2)]}{N \sqrt{d} } = \frac{1}{4N}\Big[ \E\Big[\max_{\sigma \in S_{N}(\alpha)} H_0(\sigma) \Big] - \E\Big[ \max_{\sigma \in S_N(\alpha)} H_1(\sigma) \Big]  \Big] + o_d(1). \label{eq:int1}
\end{align}
The proof of Theorem \ref{thm:cut_difference} can be completed by combining \prettyref{thm:diff_groundstate}, Lemma \ref{lem:concentration} and \eqref{eq:int1}, using a simple Borel-Cantelli argument. \qed

\subsection{Proof of \prettyref{thm:diff_groundstate}}\label{subsection:proof_diff}
For $v \in [0,1]$, consider the interpolating Hamiltonian
 \begin{align*}
 H_v(\sigma) = \sum_{ij} \frac{J_{ij}}{\sqrt{N}} (\sigma_ i - \sqrt{v} ( 2\alpha -1) )(\sigma_ j - \sqrt{v} ( 2\alpha -1)),
 \end{align*}
and  the interpolating free energy 
 \begin{align}
 F_N(v, \beta; \alpha) = \frac{1}{N} \E\Big[ \log \sum_{\sigma \in S_N (\alpha)} \exp(\beta H_v(\sigma) )\Big]. \label{eq:fn_defn}
 \end{align}
 At $v=0$ and $v=1$ these are the free energies for the Hamiltonians $H_0$ and $H_1$ respectively.
 It is convenient to make the change of variables $\sigma\mapsto \tau$ where
 \[
 \tau_i = \sigma_i - \sqrt{v}(2\alpha-1) .
 \]
Under this change of variables,  $H_v$ is a generalized mixed $p$-spin model with $\xi(t)=2t^2$ and
 where the spins take values in the set $\Sigma(v,\alpha)$, such that 
 \begin{align}
 \Sigma(v,\alpha) = \{ 1- \sqrt{v} ( 2 \alpha -1) , - 1 - \sqrt{v} (2\alpha -1) \}. \nonumber
 \end{align}
 Furthermore, if we define
  \begin{align}
 S_N(v,\alpha) = \{ \tau \in \Sigma(v,\alpha)^N : \sum \mathbf{1}\{ \tau_i = 1 - \sqrt{v} (2\alpha -1) \} = N \alpha\} \nonumber 
 \end{align}
 we can equivalently write 
 \begin{align}
 F_N(v, \beta ; \alpha) = \frac{1}{N} \E\Big[ \log \sum_{\tau \in S_N(v,\alpha)} \exp{(\beta H_v (\tau))} \Big], \nonumber 
 \end{align}
 where we abuse some notation and index the interpolating Hamiltonian by the new spins $\tau$. 
  For $\tau, \tau' \in S_N(v,\alpha)$, we define the overlap as usual 
 \begin{align}
 R(\tau, \tau') = \frac{1}{N} \sum_{i=1}^{N} \tau_i \tau_i'. \nonumber 
 \end{align}
 We define 
 \begin{align}
 T(v,\alpha) = \alpha (1 - \sqrt{v} (2 \alpha -1 ))^2 + (1 - \alpha) (1 + \sqrt{v}(2 \alpha -1 ))^2. \label{eq:t_defn} 
 \end{align}
 and note that for $\alpha \neq 1/2$ and $v > 0$, $$\tau \in S_N(v,\alpha) \iff \|\tau \|^2 = N T(v, \alpha).$$ 
 In other words, for $\alpha \neq 1/2$ and $v>0$, $S_N(v,\alpha)$ specifies sets with a constant ``self-overlap". Note that at this
 juncture, we immediately obtain Corollary \ref{cor:cut-rep}.
 
 \begin{proof}[\textbf{\emph{Proof of Corollary \ref{cor:cut-rep}}}]
This follows by combining \eqref{eq:er}, \eqref{eq:regular}, and \eqref{eq:t_defn}  with \prettyref{thm:variational_rep} in the case $v=1$ and Theorem \ref{thm:constrained-GS-SK} in the case $v=0$.
\end{proof}

 We will need the following results to complete the proof. 
  First, we obtain the following explicit expression for the derivative of the interpolating free energy in $v$ using integration by parts. To this end, we will need the following definition. For any function $f: S_N(v,\alpha) \to \mathbb{R}$, we define the expectation under Gibbs measure of the interpolated Hamiltonian 
  \begin{align}
  \langle f \rangle_v = \frac{\sum_{\tau \in S_N(v,\alpha)} f(\tau) \exp{(\beta H_v (\tau))} }{\sum_{\tau \in S_N(v,\alpha)} \exp{(\beta H_v (\tau))} }. \nonumber 
  \end{align}

 \begin{lem}
 \label{lem:int_by_parts}
 For each $N \geq 1$, and $\alpha \in (0, \frac{1}{2})$, we have, 
 \begin{align}
\frac{1}{\beta} \frac{\partial}{\partial v} F_N(v, \beta; \alpha) = - 2 \beta (2\alpha -1)^2 \frac{(1- \sqrt{v})}{\sqrt{v}} \E[ T(v) - \langle R_{12} \rangle_v]. \nonumber 
 \end{align}
 \end{lem}
   The proof will be deferred to the end of the subsection. 
 Next, we will need the following crucial theorem. 
  \begin{thm}
 \label{thm:gamma_conv_application}
For all $\alpha \in (0, \frac{1}{2})$ and $v \in (0,1)$, there exists an explicit constant $C_1(v,\alpha)>0$ such that  
 \begin{align}
 \liminf_{\beta\to \infty} \liminf_{N \to \infty} \beta \E[T^2 - \langle R_{12}^2 \rangle_v] > C_1(v,\alpha). \nonumber 
 \end{align}
 \end{thm}
The proof of this estimate will require the full machinery of the method of annealing. 
 We establish this result in Section \ref{subsec:gamma_conv_app}. 
 Before turning to these results, however,  we establish \prettyref{thm:diff_groundstate}. 
 Central to our approach is
the following well-known elementary observation  (see, e.g., \cite[Lemma 2.5]{dembo2015extremal}).
\begin{lem}\label{lemma:ground_state_approx}
The following inequality holds for any $A_N \subset \Sigma^N$ and $\beta >0$. 
\begin{align}
\Big| \frac{1}{\beta} F_{N}(\beta, \xi; A_N)  - GS_N(A_N) \Big| \leq \frac{\log |\Sigma|}{\beta}. \nonumber
\end{align}
\end{lem}

\begin{proof}[\textbf{\emph{Proof of \prettyref{thm:diff_groundstate}}}]
We observe that the thesis of \prettyref{thm:diff_groundstate} can be equivalently formulated as 
 \begin{align}
 \liminf_{N\to \infty} \liminf_{\beta \to \infty} \frac{1}{\beta} [F_N(0, \beta ; \alpha) - F_N(1, \beta; \alpha)] > C_0(\alpha) >0 . \nonumber 
 \end{align}
 
 \noindent In light of Lemma \ref{lemma:ground_state_approx}, it suffices to establish 
  \begin{align}
 \liminf_{\beta \to \infty} \liminf_{N \to \infty}  \frac{1}{\beta} [F_N(0, \beta; \alpha) - F_N (1, \beta ;\alpha)] > C_0(\alpha) . \nonumber 
 \end{align}
 To this end, we have, 
 \begin{align}
 \frac{1}{\beta} \Big[ F_N(1, \beta; \alpha) - F_N(0, \beta; \alpha) \Big] = \int_{0}^{1} \frac{1}{\beta}\frac{\partial}{\partial v} F_N(v,\beta;  \alpha) \de v. \label{eq:interpolation}
 \end{align}

 \noindent
 Using \eqref{eq:interpolation} and Lemma \ref{lem:int_by_parts}, it suffices to establish that for all $v \in (0,1)$, 
 \begin{align}
 \liminf_{\beta \to \infty} \liminf_{N \to \infty} \beta \E[ T - \langle R_{12} \rangle_v] > C_0(v, \alpha) >0  . \nonumber
 \end{align}
 Note that $T \geq | R_{12}  |$ and thus $(T - \langle R_{12} \rangle_v ) \geq (T^2 - \langle R_{12}^2 \rangle_v)/ 2 T$. Thus the proof of \prettyref{thm:diff_groundstate} follows upon applying Theorem \ref{thm:gamma_conv_application}.  
 \end{proof}
 
We establish Lemma \ref{lem:int_by_parts} in the rest of this section.

 \begin{proof}[\textbf{\emph{Proof of Lemma \ref{lem:int_by_parts}}}]
 The result will follow directly by integrating by parts \cite[Lemma 1.1]{Pan}. To see this, we begin by observing that
 \begin{align}
\frac{1}{\beta} \frac{\partial}{\partial v} F_N(v, \beta ; \alpha) = \frac{1}{N} \E[ \langle \frac{\partial}{\partial v} H_v (\sigma) \rangle_v]. \nonumber 
 \end{align}
 We note that 
 \begin{align}
 \frac{\partial}{\partial v} H_v(\sigma) = - \frac{(2\alpha -1)}{\sqrt{v}} \sum_{ij} \frac{J_{ij}}{\sqrt{N}} (\sigma_i - \sqrt{v} (2\alpha -1)). \nonumber 
 \end{align}
 Moreover, we note that for $\tau \in S_N(v, \alpha)$, $\sum_i \tau_i = N (2\alpha -1) ( 1- \sqrt{v})$. Therefore, we have, 
 \begin{align}
 \E\Big[ H_v(\sigma^1) \frac{\partial }{\partial v} H_v(\sigma^2) \Big] = - N\frac{2 ( 2\alpha -1)^2 (1 -\sqrt{v}) }{\sqrt{v}} R(\tau^1, \tau^2) . \nonumber 
 \end{align}
 Thus we have, by Gaussian integration-by-parts for Gibbs averages (see, e.g., \cite[Lemma 1.1]{Pan}), that
 \begin{equation}\label{eq:interpolation-main-step}
\frac{1}{\beta} \frac{\partial }{\partial v } F_N(v , \beta;  \alpha) = - \frac{2 \beta  ( 2\alpha -1)^2 (1 -\sqrt{v}) }{\sqrt{v}} \E[ ( T - \langle R(\tau^1, \tau^2) \rangle_v )].
 \end{equation}
 This completes the proof. 
 \end{proof}

\section{Annealed Solutions of Parisi Initial Value Problems}\label{sec:gamma-parisi-intro}
In the study of mean field spin glasses, a central role
is played by the Parisi boundary value problem, which is defined as follows.
For $\beta>0$, let $X_{\beta,T}\subset\cA_{T}$
denote the set of measures of the form 
\begin{equation}
X_{\beta,T}=\left\{ \nu=\beta\mu[0,t]dt,\quad\mu\in\Pr([0,T])\right\} . \label{eq:xbeta_defn}
\end{equation}
Equip the product space $\cA_{T}\times\R$ with the  product topology,
where we recall that $\cA_T$ was given the weak-* topology.

On this space define for every $(t,x)\in[0,T)\times \R$ 
the functional $\cF_{\beta,T}(\nu,\lambda;t,x)$ given by
\begin{equation}\label{eq:finite_temp_functional}
\cF_{\beta,T}(\nu,\lambda;t,x)=
\begin{cases}
u^\beta_{\nu,\lambda}(t,x) & \nu=\beta\mu([0,t])dt\\
+\infty & \text{otherwise},
\end{cases}
\end{equation}
where $u_{\nu,\lambda}^{\beta}(t,x)$ is the weak solution to 
the Parisi boundary value problem
\begin{equation}\label{eq:ppde-finite-beta}
\begin{cases}
\partial_{t}u+\frac{\xi''}{2}(\Delta u+\beta\mu([0,s])\left(\partial_{x}u\right)^{2})=0\\
u(T,x)=f_{\beta}(x,\lambda)
\end{cases}
\end{equation}
with boundary data 
\begin{equation}\label{eq:final-time-data-finite-beta}
f_{\beta}(x,\lambda)=
\begin{cases}
\frac{1}{\beta}\log\int_{\{\sigma^2=T\}}e^{\beta\left(\epsilon x+\lambda\epsilon^{2}\right)}d\epsilon & {T\in\{\dlower,\dupper\} } \\
\frac{1}{\beta}\log\int_{\Sigma}e^{\beta\left(\epsilon x+\lambda\epsilon^{2}\right)}d\epsilon  & \text{otherwise.} \\
\end{cases}
\end{equation}
Here $d\sigma$ is the counting measure on $\Sigma$.

A central role in our study will be played by $\cF_{\beta,T}(\nu,\lambda;t,x)$. 
More precisely, we will be interested in the limit of this functional as $\beta\to\infty$
along sequences $(\nu_\beta,\lambda_\beta)$ where $\nu_\beta\in X_\beta$.
The main issue in understanding this limit is as follows. 
On $\cA_T\times\R$, 
a typical convergent sequence,
$(\nu_\beta)$ with $\nu_\beta \in X_\beta$, satisfies
\[
\nu_\beta\to \nu = m(t)dt+c\delta_T.
\]
Thus one must have a method of interpreting $u^\beta_{\nu_\beta,\lambda_\beta}(t,x)$
in this limit. A naive approach does not suffice: in this limit
the coefficient in front of the non-linearity formally converges
to an expression of the form $f(t) = m(t)+c\delta'_T$ where
$m$ is some cadlag non-decreasing function and $\delta'_T$ is the distributional
derivative of the Dirac mass at $T$.
Evidently,  care must be taken in this limiting procedure.

To this end we introduce a notion of solution that
respects this mode of convergence called \emph{an annealed solution}. The 
main idea is that the singular contribution is,
effectively, a change of initial data. 
More precisely, we are led to the following definition.

\begin{defn}
We say that $\phi_{\nu,\lambda}(t,x)$ is an \emph{annealed solution} to \prettyref{eq:zero-temp-pde},
if for every $(t,x)\in [0,T)\times \R$, we have that
\begin{enumerate}
\item ($\liminf$ condition) If $(\nu_\beta,\lambda_\beta)\to(\nu,\lambda)$, then for $u^\beta_{\nu_\beta,\lambda_\beta}$ the solution
of \prettyref{eq:ppde-finite-beta}, we have that
\[
\liminf u^\beta_{\nu_\beta,\lambda_\beta}(t,x)\geq \phi_{\nu,\lambda}(t,x).
\]
\item ($\limsup$ condition) There is a sequence $(\nu_\beta,\lambda_\beta)\to(\nu,\lambda)$ such that for $u^\beta_{\nu_\beta,\lambda_\beta}$, the solution
of \eqref{eq:ppde-finite-beta}, we have that
\[
\lim u^\beta_{\nu_\beta,\lambda_\beta}(t,x) = \phi_{\nu,\lambda}(t,x).
\]
\end{enumerate}
\end{defn}

The goal of this section is to interpret the weak solution, $u$, of \prettyref{eq:zero-temp-pde},
as an  annealed solution. 
\begin{thm}\label{thm:annealed-soln-thm}
For every $\nu$ and $\lambda$, the weak solution to \eqref{eq:zero-temp-pde}, $u_{\nu,\lambda}(t,x)$,  is an
annealed solution.
\end{thm}

Before turning to the proof of this result, let us pause to make two comments.
\begin{rem}
We note here that the definition of annealed solution is such that the non-linear terms in \eqref{eq:pfunc-def} will $\Gamma$-converge. 
\end{rem}
\begin{rem}
Although we only use this result at the point $(0,0)${,} our argument extends to general $(t,x)$ with no change. Furthermore, working with general $(t,x)$ clarifies the proof substantially. Moreover, we imagine that this notion will be useful in future research.
Indeed, it is known \cite{AuffJag18} that the solution  $u^\beta$ and its derivatives have physical interpretations, where the values at a given point in space-time are related to Gibbs averages of natural quantities. 
We expect the notion of annealed solutions to provide
insights into the zero-temperature behavior of these physical quantities. 
\end{rem}

The proof of  \prettyref{thm:annealed-soln-thm} is in two parts. In the following, we will omit
the dependence on $(t,x)$ and $T$ whenever it is clear for readability. We
will frequently make use of basic regularity of $u^\beta$ and $u$. For the former see \cite{BAJag17} and the
latter see \prettyref{app:pde}.

\begin{proof}[\textbf{\emph{Proof of \prettyref{thm:annealed-soln-thm}}}]
We divide the proof of this theorem into two lemmas. 
We begin in \prettyref{lem:gamma-limsup} below, by proving the $\limsup$ condition.
We then turn to the $\liminf$ condition in \prettyref{lem:gamma-liminf}. 
\end{proof}

\begin{lem}[$\limsup$ condition]
\label{lem:gamma-limsup}
For every every $(t,x)\in [0,T)\times \R$, and every $(\nu, \lambda) \in \cA_{{T}} \times \mathbb{R}$, there exists a sequence $(\nu_\beta,\lambda_\beta)\to(\nu,\lambda)$
such that 
\[
\lim u^\beta_{\nu_\beta,\lambda_\beta}(t,x) = u_{\nu,\lambda}(t,x)
\]
where $u_{\nu,\lambda}$ is the weak solution to \eqref{eq:zero-temp-pde}.
\end{lem}

\begin{proof}
Our goal is to construct such a recovery sequence.  To this end, 
fix $\nu$ and $\lambda$. Let $\lambda_\beta=\lambda$.
To construct $\nu_\beta$,  we remind the reader of the 
following construction from \cite[Lemma 2.1.2]{jagTob17}. 
\begin{lem}\label{lem:recovery-sequence}
Let $\nu\in\cA_{T}$ and let $c_\beta(\nu)=c$ if $\nu(\{1\})=c>0$ and $c_\beta(\nu)=\beta^{-1}$ otherwise. For $\beta$ sufficiently large,  there exists a $q_\beta \in (0,T)$ with the following properties:
 \begin{itemize}
 \item $\int_{q_\beta}^T m(t)dt+c_\beta=\beta(T-q_\beta)$
 \item $q_\beta\to T$
 \item ${m(q_\beta)}\leq \beta$.
 \end{itemize}
 In particular, if 
 \[
 \mu_\beta= \begin{cases} m(t)/\beta & t<q_\beta \\
 1 & t\geq q_\beta
 \end{cases}
 \]
 then $d\nu_\beta=\beta \mu_\beta([0,t])dt\in X_\beta$ and $\nu_\beta\to\nu$.
 \end{lem}

Let us now show that this sequence allows us to recover the value of the function.
To this end, we solve the Parisi PDE \prettyref{eq:ppde-finite-beta}
in the interval $[q_{\beta}, T]$ using the Hopf-Cole transform. 
This yields, 
\begin{align}
u^{\beta}_{\nu_\beta,\lambda_\beta}(q_{\beta},x) 
= \frac{1}{\beta}\log\int\exp\Big[\beta\Big(\spin x + (\lambda+\frac{\beta}{2}(\xi'(T)-\xi'(q_\beta))\spin^2)\Big)\Big]d\spin.
\end{align}\label{eqn:hopf_cole_final}
Since $u^{\beta}=u^\beta_{\nu_\beta,\lambda_\beta}$ and $u=u_{\nu,\lambda}$ solve the same PDE in the time interval $(0,q_{\beta})$, their difference, $w= u^{\beta} - {u}$, solves the boundary
value problem, 
\[
\begin{cases}
\partial_t w + \frac{\xi''(t) }{2} \Big( \Delta w + m(t) (\partial_x u^{\beta} + \partial_x u ) \partial_x w \Big) = 0, \quad (t,x) \in (0, q_{\beta}) \times \mathbb{R},  \\
w(q_{\beta}, x) = u^{\beta}(q_{\beta}, x) - {u}(q_{\beta}, x) & x \in \R.  
\end{cases}
\]
Since $\partial_x u^\beta$ and $\partial_x u$ are both in $C_tC^\infty_x $ and $m(t)\in L^\infty([0,q_\beta])$ for $s<q_\beta<T$,  this is a linear heat equation and  
 we have the representation 
\begin{align}
w(t,x) &= \E_{Z_t = x} [w(q_{\beta}, Z_{q_{\beta}})], \nonumber \\
\de Z_t &= \xi''(t) m(t) (\partial_x u^{\beta} + \partial_x u) \de t  + \sqrt{\xi''(t)} \de W_t . \nonumber 
\end{align}
Thus we immediately obtain 
\begin{align}
|u^\beta(t,x) - u(t,x) | =  | w(t,x) | \leq \| u^{\beta}(q_{\beta}, \cdot) - {u}(q_{\beta}, \cdot) \|_{\infty}. \nonumber 
\end{align}
The proof of the $\limsup$ condition will be complete once we establish that the right-hand side vanishes 
as $\beta \to \infty$. 

To this end, we note that 
\begin{align}
\|u^{\beta}(q_{\beta}, \cdot) - u(q_{\beta}, \cdot) \|_{\infty} \leq \| u^{\beta}(q_{\beta}, \cdot) - u(T, \cdot) \|_{\infty} + \| u (T, \cdot) - u(q_{\beta}, \cdot) \|_{\infty}. \nonumber
\end{align}
Let us first show that by \eqref{eqn:hopf_cole_final} and \eqref{eq:zero-temp-bd}  the first term vanishes. 
Indeed, 
\begin{align*}
\abs{u^\beta(q_\beta,x)-u(T,x)} 
 &\leq \abs{\sup_{\spin}\Big\{\spin x + \Big(\lambda+(\int_{q_\beta}^T m(t)dt +c_\beta)\frac{\xi'(T)-\xi'(q_\beta)}{2(T-q_\beta)}\cdot\spin^2\Big)\Big\} - u(T,x)} +o_\beta(1)\\
 &\leq \frac{\dupper}{2} \abs{\left(\int_{q_\beta}^T m(t)dt +c_\beta\right)\frac{\xi'(T)-\xi'(q_\beta)}{T-q_\beta}-c\xi''(T)} +o_\beta(1)
 \end{align*}
where the error terms $o_\beta(1)$ are uniform in $x$. This goes to zero as $\beta\to\infty$ by \prettyref{lem:recovery-sequence}.

It remains to bound the second term. To see this,  note that by It\^o's lemma,
\begin{align}
&u(q_{\beta}, x) = \E_{\tilde{Z}_{q_{\beta}} = x} \Big[ u(T, \tilde{Z}_T) \Big], \nonumber 
\end{align}
where $Z_s$ is the It\^o process
\[
\de\tilde{Z}_s = \xi''(s) m(s) \partial_x u( s, \tilde{Z}_s) \de s + \sqrt{\xi''(s)} \de W_s .
\]
Thus since $u$ is Lipschitz in space uniformly in time,
\begin{align}
\| \tilde{u}_{\lambda, \nu}(q_{\beta}, \cdot) - \tilde{u}_{\lambda, \nu}(T, \cdot) \|_{\infty} \leq \| \partial_x \tilde{u}_{\lambda, \nu} \|_{\infty} \E| \tilde{Z}_T - \tilde{Z}_{q_{\beta}} | \to 0 \nonumber
\end{align}
as $\beta \to \infty$ using elementary properties of diffusions, as $q_{\beta} \to T$ as $\beta \to \infty$. 
\end{proof}

We now turn to the $\liminf$-condition.
For this we recall the dynamic programming  formulation 
of these PDEs.
Let $\cB_t^T$ be the space of all bounded processes on $[t,T]$ that are progressively measurable
with respect to the filtration of Brownian motion. In the following, for a measure $\mu$, we let $\mu(s)=\mu([0,s])$.

\begin{lem}\label{lem:dpp}
The weak solution $u_{\nu,\lambda}^\beta$ of \prettyref{eq:ppde-finite-beta} {corresponding to $\nu = \beta \mu(s)ds$ }  solves
\begin{equation}\label{eq:dpp-finite-beta}
u^\beta_{\nu,\lambda}(t,x) = \sup_{\alpha\in \cB_t^T} \E_{X_t^\alpha =x} \left( f_\beta(X_T^\alpha,\lambda) - \frac{\beta}{2}\int_t^T\xi''(s){\beta\mu(s)}\alpha^2_sds\right)
\end{equation}
where $X^\alpha$ solves
\[
dX^\alpha_s = \xi''(s){\beta\mu(s)}\alpha_sds+\sqrt{\xi''(s)}dW_s
\]
with initial data $X_t=x$. Furthermore, any optimal control, $\alpha^*$, satisfies
\[
{\mu(s)}\alpha_s^* ={\mu(s)}\partial_x u_{\nu,\lambda}^\beta(s,X_s) \quad a.s.
\]
where $X_s$ solves:
\begin{equation}\label{eq:local-field}
dX_s = \xi''(s){\beta \mu(s)}\partial_x u_{\nu,\lambda}^\beta(s,X_s)ds+\sqrt{\xi''(s)}dW_s.
\end{equation}

Furthermore, the weak solution $u$ of \prettyref{eq:zero-temp-pde} {corresponding to $\nu = m(s)ds+c\delta_T$ } solves
\begin{align}\label{eq:dpp-zero-temp}
u_{\nu,\lambda}(t,x) = \sup_{ \gamma \in \mathcal{B}_t^T} \E_{\tilde{X}_t^{\gamma} = x } \Big[ u_{\nu,\lambda}(T, \tilde{X}_T^{\gamma}) - \frac{1}{2} \int_t^{T} \xi''(s) m(s) \gamma^2_s \de s \Big], 
\end{align}
where $\tilde X^{\gamma}$ has initial data $\tilde{X}^{\gamma}_t =x$ and solves the SDE
\begin{align}
\de \tilde X^{\gamma}_s = \xi''(s) m(s) \gamma_s \de s + \sqrt{\xi''(s)} \de W_s. \nonumber
\end{align}
\end{lem}

This result is an immediate consequence of the verification argument \cite{FlemRishOptControl}. 
For $u^\beta$, this is immediate  after recalling that $u$ is smooth in space and weakly differentiable in time, with bounded derivatives \cite[Appendix A]{BAJag17}. 
See \cite[Lemma 18]{jagannath2015dynamic} for a proof that applies essentially unchanged to our setting.
For $u$, the same proof applies after observing that the solution has the regularity from \prettyref{lem:ppde-reg-zero-temp}.
The only point to note is that one should apply It\^o's lemma for times $t<T$ and then pass to a limit as $t\to T$. 
Similar arguments appear frequently in the literature, see, e.g., \cite{ AuffChen14,auffinger2017energy,bovklim09,jagannath2015dynamic}, so we omit it.
Before turning to the lower bound, note the following.
\begin{lem}\label{lem:psi-diff}
Consider the function 
\[
\psi(x,y)=\max_{\spin\in  \Sigma}\Big[ \spin x+\spin^{2}y \Big].
\]
For every $y$, this function is differentiable in $x$ Lebesgue a.e.
Furthermore, the derivative is continuous in $x$ Lebesgue a.e. and is given by 
\[
\partial_{x}\psi(x,y)=\spin_{*},
\]
where $\spin_{*}$ is the unique solution to $\psi(x,y)=\spin_{*}x+\spin_{*}^{2}y$.
\end{lem}
\begin{proof}
Fix $y$. Then $x\mapsto\psi(x,y)$ is convex in $x$. Thus by Alexandrov's theorem \cite{evansgariepy2015measure}, it is differentiable in $x$ Lebesgue a.e. with a derivative
that is continuous in $x$ Lebesgue a.e.. Combining this with Danskin's
envelope theorem \cite{bernhard1995theorem}, we see that the derivative is given by
\[
\partial_{x}\psi(x,y)=\spin_{*},
\]
where $\spin_{*}$ is such that $\psi(x,y)=\spin_{*}x+\spin_{*}^{2}y$.
 (Implicit in this argument is that there is a unique such $\spin_{*}$
for almost every $x$, by another application of Alexandrov's theorem.)
\end{proof}
With these in hand we can now prove the lower bound. 
\begin{lem}[$\liminf$-condition]
\label{lem:gamma-liminf}
For every $(t,x)\in [0,T)\times \R$ and every $(\nu, \lambda) \in \cA \times \mathbb{R}$, if $(\nu_{\beta} ,\lambda_{\beta}) \in \cA \times \mathbb{R}$ with $(\nu_{\beta}, \lambda_{\beta}) \to (\nu, \lambda)$, then
\[
\liminf u^\beta_{\nu_\beta,\lambda_\beta}(t,x)\geq u_{\nu,\lambda}(t,x)
\]
where $u_{\nu,\lambda}$ is the weak solution to \eqref{eq:zero-temp-pde}.
\end{lem}

\begin{proof}
We focus on the case $T\in(\dlower,\dupper)$. The case $T\in\{\dlower,\dupper\}$ is the same
 after taking $\lambda=0$ everywhere, as in this case the functional is constant in $\lambda$. 

 Let 
\[
\phi(x) = \textrm{argmax}_{\spin \in\Sigma} \Big[\spin x + \Big( \lambda + \frac{c}{2} \xi''(T) \Big)\spin^2  \Big].
\]
Observe that by \prettyref{lem:psi-diff}, $\phi \in L^\infty(\R)$. Thus for any $M>0$, we can consider a sequence of smooth, 
compactly supported function $\phi_n$
such that $\phi_n\to f$ strongly in $L^2([-M,M])$ and almost everywhere, where $f$ is as in \eqref{eq:zero-temp-bd}.
Furthermore, one may take $\phi_n$ such that
\[
\norm{\phi_n}_\infty\leq C
\]
for some $C>0$ that does not depend on $M$ or $n$. 
Fix $\gamma \in \mathcal{B}_0^{T}$,  $\tau\in(t,T)$, and for each $n$, consider the control
\begin{align}
\alpha_{s}^{\tau,n} = \begin{cases}
\gamma_s & {\textrm{if}}\, \,s \leq \tau \\
 \phi_n(\tilde{X}^{\gamma}_s) & {\textrm{if}}\,\, s > \tau
\end{cases}. \nonumber
\end{align}
Note that $\alpha_s^{\tau,n}$ is a bounded, progressively measurable control. 
Further, the continuity of $\phi_{n}$ ensures the left-continuity of $\alpha^{\tau,n}_s$ as $s \uparrow T$.

By topological properties of $\cA_T$, see \prettyref{lem:convergence},  we have that for each $n$,
\begin{align}
\int_{t}^{T} \xi''(s) (\alpha_s^{\tau,n})^2 \de \nu_{\beta}(s) 
&\stackrel{\beta \to \infty}{\to} \int_{t}^{\tau} \xi''(s) \gamma_s^2 m(s) \de s +  \int_{\tau}^{T} \xi''(s) \Big( \phi_{n}(\tilde{X}^{\gamma}_s) \Big)^2\de\nu(s)=: D_{\tau,n}(\nu)  \nonumber
\end{align}
We then let $\tau \uparrow T$ through the continuity points of $\nu$ to derive 
$$D_{\tau}(\nu) \to \int_{t}^{T} \xi''(s) \gamma_s^2 m(s) \de s + c (\phi_n(\tilde{X}^{\gamma}_T) )^2 \xi''(T)$$ 
almost surely.   

Similarly, we have that
\begin{align}
\lim_{\beta\to\infty}X_T^{\alpha^{\tau,n}} 
&\to \int_{0}^{\tau} \xi''(s) \gamma_s m(s) \de s + \int_{\tau}^{T} \xi''(s) \phi_{n}(\tilde{X}^{\gamma}_s) \de{\nu}(s) + \int_{0}^{T} \sqrt{\xi''(s)} \de W_s + x =: Z_T(\tau,\nu), \nonumber
\end{align}
almost surely.
As before, we let $\tau \uparrow T$ to obtain, 
$$Z_T(\tau, \nu) \to \tilde{X}_T^{\gamma} + c \phi_n(\tilde{X}^{\gamma}_T) \xi''(T).$$
By the dynamic programming principle \eqref{eq:dpp-finite-beta},
\begin{align*}
\liminf_{\beta \to \infty} u^\beta_{\nu_\beta,\lambda_\beta}(t,x)
\geq \liminf_{\beta \to \infty} \E_{X_t^{\alpha^{\tau,n}}=x} \Big[f_{\beta}(X_{T}^{\alpha^{\tau,n}},\lambda_\beta) - \frac{1}{2} \int_{0}^{T} \xi''(s) (\alpha_s^{\tau,n})^2 \de\nu_{\beta}(s) \Big]. 
\end{align*}
If we combining these results with the fact that 
\begin{align*}
f_{\beta}(x,\lambda) \geq f(x,\lambda,0), 
\end{align*}
we may send $\beta \to \infty$, followed by $\tau \uparrow T$ and use the dominated convergence theorem to conclude that 
\begin{align*}
\liminf_{\beta \to \infty} u^\beta_{\nu_\beta,\lambda_\beta}(t,x) &\geq \liminf_{\beta \to \infty} \E_{X_t^{\alpha^{\tau,n}}=x} \Big[ \sup_{\spin \in \Sigma } \Big[ X_{T}^{\alpha^{\tau,n}} \spin + \lambda_{\beta} \spin^2 \Big] - \frac{1}{2} \int_0^T \xi''(s) (\alpha_s^{\tau,n})^2 \de\nu_{\beta}(s) \Big] 
\geq \E[\zeta_n], \nonumber
\end{align*}
where we define 
\begin{align}
\zeta_n= \Big[ \sup_{\eps\in\Sigma}\Big(\tilde{X}_T^{\gamma} + c \phi_n(\tilde{X}^{\gamma}_T) \xi''(T))\eps+  \lambda\eps^2\Big) - \frac{1}{2} \int_{0}^{T} \xi''(s) \gamma_s^2 m(s) \de s - \frac{c}{2} (\phi_n(\tilde{X}^{\gamma}_T) )^2 \xi''(T) \Big]. \nonumber
\end{align}
Finally, it remains to send $n \to \infty$. We observe that 
\begin{align}
\E[\zeta_n] = \E[\zeta_n \mathbf{1}(\tilde{X}^{\gamma}_T \in [-M, M] ) ] + \E[\zeta_n \mathbf{1} (\tilde{X}^{\gamma}_T \notin [-M,M] )]. \label{eqn:intermediate_localization}
\end{align}
Conditionally on the event $\tilde{X}^{\gamma}_T \in [-M, M]$, $\phi_n(\tilde{X}^{\gamma}_T) \to \phi(\tilde{X}^{\gamma}_T)$ almost surely as $n \to \infty$. We define 
\begin{align}
\zeta_{\infty} = \sup_{\Sigma}\Big[ \Big( \tilde{X}_T^{\gamma} + c \phi(\tilde{X}^{\gamma}_T) \xi''(T)  \Big)\spin + \lambda \spin^2  \Big] - \frac{1}{2} \int_{0}^{T} \xi''(s) \gamma_s^2 m(s) \de s - \frac{c}{2} (\phi(\tilde{X}^{\gamma}_T) )^2 \xi''(T). \nonumber
\end{align}
By the dominated covergence theorem, we may send first $n\to\infty$ and then $M\to\infty$ to obtain,
\begin{align}
\lim_{M \to \infty} \lim_{n\to \infty} \E[\zeta_n \mathbf{1}(\tilde{X}^{\gamma}_T \in [-M, M]) ] = \E[\zeta_{\infty}]. \nonumber
\end{align}
Similarly, since $\phi_n$ and $\gamma$ are bounded and $\lambda,\nu$ are fixed, we see
that $\zeta_n$ is uniformly bounded, so that 
by H\"olders inequality,
\[
\abs{\E[\zeta_n\mathbf{1} (\tilde{X}^{\gamma}_T \notin [-M,M] ) ] }\leq  \norm{\zeta_n}_\infty \P[\tilde{X}^{\gamma}_T \notin [-M, M]].
\]
We let $M \to \infty$ to control this term and obtain the lower bound
\begin{align}
\liminf_{\beta \to \infty} u^\beta_{\nu_\beta,\lambda_\beta}(t,x)  \geq \E[\zeta_{\infty}]. 
\end{align}
Finally, after recalling the choice of $\psi$  and maximizing in $\gamma$,
a  direct computation yields that 
$$\sup_\gamma \E[\zeta_{\infty}] = u_{\nu,\lambda}(t,x).$$
The result then follows. 
\end{proof}

\section{Local Free energies: their derivatives and zero temperature limits}\label{sec:loc-free-energy}

Our starting point for this section is equation  \prettyref{eq:loc-fe}.
Following \cite{panchenko2005generalized}, we refer to $F(\beta,\xi;T)$ as the \emph{local} free energy. 
The functional $\cP_{\beta,T}:\cA_{T}\times\R\to\R$ in \prettyref{eq:loc-fe} 
 is called the local Parisi functional  and is
given by 
\begin{equation}\label{eq:pfunc-def}
\cP_{\beta,T}(\nu,\lambda)=\cF_{\beta,T}(\nu,\lambda;0,0)-\lambda T-\frac{\beta}{2}\int_{0}^{T}\xi''(s)s\mu([0,s])ds,
\end{equation}
where $\cF_{\beta,T}$ is defined in \prettyref{eq:finite_temp_functional}. {In particular, recall that $\cF_{\beta,T}$ is infinite on $\cA_T\setminus X_{\beta,T}$, so that $\cP_{\beta,T}$ is infinite as well. }
In this section, we study basic properties of this functional.  
First, in \prettyref{sec:diff-fe}, we compute the derivative of the variational problem \prettyref{eq:loc-fe} in $\beta$ and relate this to the minimizer of $\cP_{\beta,T}$. In \prettyref{sec:gamma-conv-local-fe},
we then compute the $\Gamma$-limit of the sequence $(\cP_{\beta,T})_\beta$ as $\beta$ tends to infinity and establish a  precompactness theorem regarding the corresponding sequence of minimizers. We end this section by applying this $\Gamma$-convergence result to prove \prettyref{thm:variational_rep}.

\subsection{Differentiability of the Local Free energy} \label{sec:diff-fe}

We begin by showing that \prettyref{eq:loc-fe} is differentiable in $\beta$
and to provide an expression for this derivative in terms of the minimizer of the 
variational problem. For the purposes of this section, it is convenient to restrict
$\cP_{\beta,T}$ to the space $X_{\beta,T}$, defined in \eqref{eq:xbeta_defn}. Furthermore, it is convenient to 
view this as a function of $\mu$ where $\nu =\beta \mu([0,s])ds$.
To this end, we denote 
\[
P_{\beta,T}(\mu,\lambda)= \beta\cP_{\beta,T}(\nu,\lambda/\beta) = v_{\mu,\lambda}(0,0)-\lambda T-\frac{\beta^2}{2}\int_0^T\xi''(s)s\mu([0,s])ds
\]
where $\nu = \beta\mu([0,s])ds$ and $v_{\mu,\lambda}(t,x)=\beta u_{\nu,\lambda/\beta}(t,x/\beta)$, where $u^{\beta}$ solves \eqref{eq:ppde-finite-beta}.  In this notation,
\[
F(\beta,\xi;T)= \inf_{\mu,\lambda}P_{\beta,T}(\mu,\lambda)
\]
When it is clear, we will also omit dependence on $T$ in our notation.
The main result of this section is the following. 

\begin{thm}\label{thm:beta-deriv-overlap}
We have that
\[
\partial_\beta F(\beta,\xi;T) = \beta\int (\xi(T)-\xi(t) )d\mu
\]
where for $T>0$, $\mu$ is the unique minimizing measure of $P_{\beta,T}$.
\end{thm}
\begin{rem}
In the case that $\Sigma=\{\pm1\}$, this result was obtained earlier in \cite{Panch08}.
An alternative proof was provided in \cite{auffinger2016legendre}. Both of these 
arguments use in an essential way that the initial data satisfies
$(\partial_x f)^2 +\partial_x^2 f = const.$ which does not hold in the general setting. 
\end{rem}
\begin{rem}
The results of  \cite{panchenko2005generalized} hold in the more general
setting that $\Sigma$ is an arbitrary compact set and where the integral in this definition with respect to any measure $\rho$ with
support $\Sigma$. 
This result also extends to this setting with no change, however
for the sake of exposition and with an eye toward our eventual application,
we do not consider this setting here. 
\end{rem}

We begin the proof of \prettyref{thm:beta-deriv-overlap} with the following lemmas.  
The first result is regarding the convexity of $P_{\beta,T}$ and the uniqueness (or non-uniqueness) of minimizers.
This
was first proved in \cite{auffinger2014parisi}, see also \cite{bovklim09,jagannath2015dynamic}.
\begin{lem}\label{lem:strict-convexity-finite-beta}
If $T\in (\dlower,\dupper)$, the map $P_{\beta,T}$ is strictly convex in $(\lambda,\mu)$.
If $T\in\{\dlower,\dupper\}$, and $\abs{\{\sigma^2=T\}}=2$, the map is strictly convex in $\mu$.
If $T=0$, the map is identically zero. 
Otherwise the map is convex and the functional is uniquely minimized at $\mu=\delta_0$. 
\end{lem}
\noindent We defer the proof of this to the end of the section.

We note here that since $u_{\nu,\lambda}$ has the representation \eqref{eq:dpp-finite-beta}, $v_{\mu,\lambda}$ has the representation
\[
v_{\mu,\lambda}(t,x) = \sup_{\alpha\in\cB_T} \E\left[ \beta f_\beta (X^\alpha_T/\beta,\lambda/\beta) - \frac{\beta^2}{2}\int_t^T \xi''(s)\mu(s)\alpha_s^2ds\right]
\]
with an optimal control 
\[
\alpha^*_s = \partial_x v_{\mu,\lambda}(t,x)
\]
and optimal trajectory $\hat X_s$ which solves
\[
d\hat{X}_s = \beta^2\xi''(s)\mu(s)\partial_x v_{\mu,\lambda}(t,x)ds + \sqrt{\beta^2\xi''(t)}dW_t.
\]

Consequently, we have the following characterization 
of minimizers of $P_{\beta,T}$. In the following let
\begin{equation}
G_{\mu,\lambda}(t)=\int_t^T\beta^2\xi''(s)(\E (\partial_x v_{\mu,\lambda})^2(s,\hat{X}_s) - s)ds.
\end{equation}

\begin{lem}\label{lem:optimality-finite-beta}
If $T\in(\dlower,\dupper)$, there is a unique minimizing pair $(\mu,\lambda)$. Furthermore this pair 
satisfies 
\begin{align}
\mu(G(s)=\min_{t\in[0,T]} G(t)) & = 1,\label{eq:G-minimizing}\\
\partial_\lambda v_{\mu,\lambda}(0,0) &= T.\label{eq:fixed-point-T}
\end{align} 
Finally  we have that for every $q$ in the support of the
optimizing $\mu$, 
\begin{equation}\label{eq:fixed-point-q}
\E \partial_x v_{\mu,\lambda}(q,\hat{X}_q)^2 = q.
\end{equation}
If $T=\dupper$, or $T=\dlower$ with $\dlower>0$ then the above still holds, with the exception of \prettyref{eq:fixed-point-T}.
If $T=0$, all measures minimize and the above conditions are vacuous. 
\end{lem}
\begin{rem}
Note that in the case $T\in\{\dlower,\dupper\}$, if we instead work with the definition of $f_\beta(x,\lambda)$ that
does not change in $T$, we see that \prettyref{eq:fixed-point-T} still holds, except in a limiting sense.
\end{rem}
\noindent We defer this proof to the end of the section as well. 

In the subsequent discussion we will be repeatedly differentiating an optimization problem. These
results generally go by the name of  ``envelope theorems''. The two envelope theorems we shall use
are Danskin's theorem \cite{bernhard1995theorem}, as well as the following lemma from \cite{Chen15}.

\begin{lem}\label{lem:wei-kuo}
Let $K$ be a metric space and $I = [a,b)$ be a half open interval. Let $f$ be a real valued
function on $K\times I$ and $g(y)=\sup_{x\in K} f(x,y)$. Suppose that there is a $K$-valued
continuous function $a(y)$ on $I$ such that $g(y)=f(a(y),y)$ and $\partial_y f$ is jointly continuous
on $K\times I$. Then $g$ is right differentiable with derivative $\partial_y f(a(y),y)$
\end{lem}
In the following, it is useful to equip the space $\cB_T$ with the $L^2$ metric and $\Pr([0,T])$
with the metric $d(\mu,\nu) = \int \abs{\mu([0,t])-\nu([0,t])}dt$ which metrizes the weak-* topology on this set.

Let us now prove the main result of this section.
\begin{proof}[\textbf{\emph{Proof of \prettyref{thm:beta-deriv-overlap}}}]
The following argument is a modification of \cite[Prop. 4]{auffinger2016legendre}.
Fix $\beta,\lambda$. For ease of notation, let $v=v_{\mu,\lambda}$, and similarly for $u$ when it is unambiguous, and also let $g(x)=\beta f_\beta(x,\lambda/\beta)$. Let us take $T>0$ as the case $T=0$ is vacuous, by \prettyref{lem:strict-convexity-finite-beta}.
Since $u^{\beta}$ has the variational representation  \prettyref{eq:dpp-finite-beta},
we may apply \prettyref{lem:wei-kuo} to obtain
\begin{align*}
\partial_\beta v(t,x) &= \E \partial_x v(T,\hat{X}_T)\left(2\beta \int_t^T\mu(s)\xi''(s)\partial_x v(s,\hat{X}_s)ds+\int_t^T \sqrt{\xi''}dW_s\right)\\
	&\qquad\qquad- \beta \int_t^T\mu(s)\xi''(s)\E(\partial_x v(s,\hat{X}_s))^2ds.
\end{align*}
(We use here, implicitly, that the map $\beta\mapsto v^\beta$ is continuous.)
Since $w=\partial_x v$ weakly solves the heat equation
\[
\partial_t w+L w =0
\]
with initial data $w(T,\cdot)=\partial_x g$, where $L=\beta^2\frac{\xi''}{2}(\partial_x^2 +2\mu \partial_x v \partial_x)$
is the infinitesimal generator of  \prettyref{eq:local-field}, it follows that $\partial_x v(s,\hat{X}_s)$
is a martingale:
\[
\partial_x v(s,\hat{X}_s)=\int_0^s \partial_x^2 v(s',\hat{X}_{s'})\sqrt{\beta^2\xi''(s')}dW_{s'}+\partial_x v(0,0).
\]
Thus 
\[
\partial_\beta v(0,0)=\beta\left\{ \int_0^T\mu(s)\xi''(s)\E (\partial_x v)(s,\hat{X}_s)^2ds + \int_0^T\xi''(s)\E \partial_x^2 v(s,\hat{X}_s)ds\right\}.
\]
Integrating the second term by parts,  using It\^o's lemma and the fact that $\xi'(0)=0$, we
see that
\[
 \int _0^T \xi''(s)\E \partial_x^2v(s,\hat{X}_s)ds = \xi'(T) \E \partial_x^2 v(T,\hat{X}_T) + \beta^2\int_0^T \mu(s) \xi'(s)\xi''(s) \E (\partial_x^2 v(s,\hat{X}_s) )^2ds.
\]
Integrating the first term by parts and applying It\^o's lemma again, we see that 
\begin{align*}
\int_0^T\mu(s) \xi''(s) \E (\partial_x v)^2(s,\hat{X}_s)ds &= \xi'(T)\E (\partial_x v)^2(T,\hat{X}_T)-\int_0^T \xi'(s)\E (\partial_x v)^2(s,\hat{X}_s)d\mu\\
&\qquad -\beta^2\int_0^T \mu(s) \xi'(s) \xi''(s) \E (\partial_x^2 v(s,\hat{X}_s))^2 ds.
\end{align*}
Combining these results, we obtain
\[
\partial_\beta v(0,0) = \beta\left\{ \xi'(T)\E\left( \partial_x^2 v(T,\hat{X}_T)+(\partial_x v)^2(T,\hat{X}_T)\right) - \int_0^T\xi'(s)\E(\partial_x v)^2(s,\hat{X}_s)d\mu\right\}.
\]
Differentiating $P_{\beta,T}$ in $\beta$, applying \prettyref{eq:fixed-point-q}, and integrating by parts, we see that at the optimal $\mu_*$, 
\[
\partial_\beta P_{\beta,T}(\mu_*,\lambda) = \beta\left\{\xi'(T)\left[\E\left( \partial_x^2 v_{\mu_*,\lambda}(T,\hat{X}_T)+(\partial_x v_{\mu_*,\lambda})^2(T,\hat{X}_T)\right) -T\right ] +\int_0^T\xi(T)-\xi(t)d\mu_{*}(t)\right\}.
\]
Let us now show that at $(\mu_*,\lambda_*)$,
\begin{equation}\label{eq:combines-to-T}
\E\left( \partial_x^2 v_{\mu_*,\lambda_*}(T,x)+(\partial_x v_{\mu_*,\lambda_*})^2(T,x)\right) =T.
\end{equation}
To this end, observe that if we define $\pi_{x,\lambda}(d\spin)$ by 
\[
\pi_{x,\lambda}(d\spin)=\frac{e^{\spin x + \spin^2 \lambda}}{\int e^{\spin x+\spin^2 \lambda}d\spin}d\spin,
\]
then we have
\begin{equation}\label{eq:f-derivs}
\partial_x g = \gibbs{\spin}, \quad \partial_x^2 g=\gibbs{\spin^2}-\gibbs{\spin}^2, \quad\text{ and }\quad 
\partial_\lambda g =\gibbs{\spin^2}
\end{equation}
where $\gibbs{\cdot}$ denotes expectation with respect to $\pi_{x,\lambda}$. Thus
\[
\partial_x^2 g + (\partial_x g)^2 = \partial_\lambda g.
\]
In the case that $T=\dlower$ or $T=\dupper$, if we take $\lambda$ in the above expression to either $-\infty$ or 
$+\infty$ respectively, then we obtain the desired expression. 

It remains to consider the case $T\in(\dlower,\dupper)$.
By a classical differentiable dependence argument  (see, e.g., \cite[Appendix A]{BAJag17}),
$v$ is classically differentiable in $\lambda$, and $w=\partial_\lambda v$
weakly solves the Cauchy problem
\[
\partial_t w + L w = 0
\] 
with boundary data $w(T,\cdot)=\partial_\lambda g$. Thus 
\begin{equation}\label{eq:v-diff}
\partial_\lambda v_{\mu_*,\lambda_*}(0,0) = \E \partial_\lambda \bar g(\hat X_T),
\end{equation}
where $\bar g$ is $g$ evaluated at $\lambda=\lambda_*$.
As a result \prettyref{eq:combines-to-T} then follows from \prettyref{eq:fixed-point-T} at $\lambda =\lambda_*$, 
the optimal $\lambda$.

Combining these results with Danskin's envelope theorem \cite{bernhard1995theorem}, we see that
\[
\partial_\beta F = \partial_\beta P_{\beta,T}(\mu_*,\lambda_*)= \beta \int_0^T \xi(T)-\xi(t)d\mu_{{*}}(t)
\]
as desired.
\end{proof}

\begin{proof}[\textbf{\emph{Proof of \prettyref{lem:strict-convexity-finite-beta}}}]
Let us begin with the first case $T\in(\dlower,\dupper)$. The proof in the case $T\in\{\dlower,\dupper\}$ and $\abs{\{\sigma^2=T\}}=2$
is identical. This proof is verbatim the argument in \cite{jagannath2015dynamic}.

Take $(\mu_0,\lambda_0)$,$(\mu_1,\lambda_1)$ distinct in $\Pr([0,T])\times \R$. 
Let $\mu_\theta = \theta\mu +(1-\theta)\nu$ and define $\lambda_\theta$ analogously.
 Let $\alpha^\theta$ be the optimal control for the Parisi PDE corresponding to $(\mu_\theta,\lambda_\theta)$.
 Consider the processes $Y_t$ and $Z_t$ which solve
 \[
 dY_t = \beta^2\xi''(t)\mu_0(t)\alpha_t^\theta dt+\sqrt{\beta^2\xi''(t)}dW_t \quad\text{ and }\quad dZ_t =\beta^2\xi''(t)\mu_1(t)\alpha_t^\theta dt+\sqrt{\beta^2\xi''(t)}dW_t,
 \]
with initial data $Z_0=Y_0=0$. 

Therefore by the strict convexity of $g(\lambda,x) = \beta f_\beta(x/\beta,\lambda/\beta)$ in the pair $(\lambda,x)$, 
 and the dynamic programming principle \prettyref{eq:dpp-finite-beta},
 \begin{align*}
 v_{\mu_\theta,\lambda_\theta}(0,0) &= \E\left[g(\lambda_\theta,\hat{X}_T^{\alpha^\theta})-\frac{\beta^2}{2}\int_0^T \xi''(s)\mu_\theta(s)(\alpha_s^\theta)^2ds \right]\\
 & \leq \theta\E\left[g(\lambda_1,Z_T)  - \frac{\beta^2}{2}\int_0^T \xi''(s)\mu_1(s)(\alpha_s^\theta)^2ds\right]\\
 &\qquad +(1-\theta)\E\left[ g(\lambda_0,Y_T) -  \frac{\beta^2}{2}\int_0^T \xi''(s)\mu_0(s)(\alpha_s^\theta)^2ds\right]\\
 & \leq \theta v_{\mu_1,\lambda_1}(0,0)+(1-\theta)v_{\mu_0,\lambda_0}(0,0).
 \end{align*}
 Furthermore the first inequality is strict if either $\lambda_1\neq\lambda_0$ or $P(Y_T\neq Z_T)>0$.
 Thus it remains to show that this probability is positive provided $\mu_0\neq\mu_1$. 
 
To this end, observe that it suffices to show that 
 \[
\Var(Y_T-Z_T)>0.
 \]
By It\^o's lemma,
 \[
 \Var(Y_1-Z_1)=\int_{[0,T]^2} \Delta_s\Delta_t K(s,t)ds\,dt,
 \]
 where $\Delta_s = \xi''(s)(\mu_0(s)-\mu_1(s))$ and 
 \[
 K(s,t)=\E\left[(\alpha_s^\theta-\alpha_0^\theta)\cdot(\alpha_t^\theta-\alpha_0^\theta)\right],
 \]
so it suffices to show that $K$ is positive definite.  By It\^o's Isometry,
 \[
 K(s,t)=p(t\wedge s)
 \] 
 where 
 \[
 p(s) =\int_0^s\beta^2\xi''(t)\E\partial_x^2 v(t,X_t^{\theta})dt.
 \]
 By \prettyref{eq:f-derivs}, $\partial_x^2 f>0$. Thus by It\^o's lemma,
 we have the maximum principle: 
 \begin{equation}\label{eq:max-princ}
 \partial_x^2 v(t,x) = \E_{X_t=x}\left( \partial_x^2 f(X_T)+ \int_t^T\xi''(s)\mu(s)\partial_x^2v(s,X_s)ds\right)>0 .
 \end{equation}
 This immediately implies that $p$ is strictly increasing, so that $K$ is positive-definite as desired. 
 
 In the remaining case, $\abs{\sigma^2=T}=1$, and one can explicitly solve the PDE to find that
 \[
 P_{\beta,T}(\mu,\lambda)=\frac{\beta^2}{2}\int_0^T \xi''(s)\mu(s)(T^2-s)ds
 \]
 which is evidently maximized at $\mu= \delta_0$ and uniquely so if and only if $T^2>0$.
 \end{proof}
In the subsequent it will be useful to define the following log-moment 
generating function,
\begin{equation}\label{eq:psi}
\psi(\theta)=\log\int e^{\theta \spin^2}d\spin.
\end{equation}
Observe that, $\psi$ is continuous and monotone, with 
\begin{equation}
\begin{aligned}\label{eq:mgf-limits}
\lim_{\theta\to\infty}\psi(\theta)/\theta  =\dupper, \,\,\,\,\,
\lim_{\theta\to-\infty}\psi(\theta)/\theta  =\dlower.
\end{aligned}
\end{equation}
Note that in the case that $\dupper=\dlower$, $\psi$ is constant.
\begin{lem}\label{lem:mgf-bound}
For every $\mu,\lambda,\beta,\xi,$ and $T\in(\dlower,\dupper)$, we have that
\[
v_{\mu,\lambda}\geq \psi( \lambda).
\]
\end{lem}
\begin{proof}
By the parabolic comparison principle (\prettyref{lem:ppde-reg-zero-temp}), we see
that 
$v_{\mu,\lambda}(t,x)\geq v_{\delta_{T},\lambda}(t,x)$. Thus
\[
v_{\mu,\lambda}(0,0)\geq\E\log\int e^{(\spin B_{\beta^2\xi'(T)}+\lambda\spin^{2})}d\spin,
\]
where $B_{t}$ is a standard Brownian motion run until time $t.$
Using again the convexity of the map 
\[
x\mapsto\log\int e^{\spin x+\lambda\spin^{2}}d\spin,
\]
(see \prettyref{eq:f-derivs}) we have that this is lower bounded by 
\[
v_{\mu,\lambda}(0,0)\geq \log\int e^{\lambda\spin^{2}}d\spin=\psi(\lambda)
\]
as desired.
\end{proof}

\begin{proof}[\textbf{\emph{Proof of \prettyref{lem:optimality-finite-beta}}}]
That a minimizer is unique follows by \prettyref{lem:strict-convexity-finite-beta}. 
That a minimizer exists can be seen as follows.
Firstly, $\mu\in\Pr([0,1])$ which we may equip with the weak-* topology. 
Thus it suffices to show that we may restrict $\lambda$ to a compact set if $T\in(\dlower,\dupper)$,
as in the other case we may take $\lambda=0$.
To this end, suppose first that $T\in(\dlower,\dupper)$. 
Observe that for any $\mu$, we then have that by \prettyref{lem:mgf-bound}
\[
P_{\beta,T}(\mu,\lambda)\geq \psi(\lambda)-\lambda T-\int s\xi''(s)ds.
\]
Thus the limit of the right hand side as $\lambda\to\pm\infty$ is infinity. Thus we may restrict
to a compact set. The result then follows by (lower semi)continuity of $P_{\beta,T}$.

Let us now prove  \prettyref{eq:G-minimizing}{. }
Let $\gamma_t=(\mu_t,\lambda)$ be a path such that $\mu_t$ 
is weakly differentiable on $(0,1)$ and right weakly differentiable at $t=0$ in the sense that
\[
\lim_{t\to 0^+}\frac{\mu_t-\mu}t = \dot{\mu}
\]
weak-* as measures for some signed measure $\dot{\mu}$. 

By the same argument as in \cite[Lemma 3.2.1]{JagTobPD15},
we have that $P_{\beta,T}(\mu_t,\lambda)$ is right differentiable at $t=0$, and 
\[
\frac{d}{dt}^+P_{\beta,T}(\mu_t,\lambda)=\int G(t) d\dot{\mu}.
\]
The only difference is to notice that since $\partial_x f$ from is uniformly bounded in $(x,\lambda)$ by \eqref{eq:f-derivs}, $\partial_x u$ is uniformly bounded in
$(t,x,\lambda)$ as well, using \prettyref{lem:ppde-reg-zero-temp}.
Thus by the first order optimality conditions for  convex functions,
\[
\int G(t) d \dot{\mu}\geq0
\]
for all such paths. Taking $\dot\mu = \mu_1-\mu_0$ yields \prettyref{eq:G-minimizing}. To obtain \eqref{eq:fixed-point-T}, we differentiate
the variational formula in $\lambda$ and use that $v$ is classically differentiable in $\lambda$ as explained above \eqref{eq:v-diff}.

It remains to prove \prettyref{eq:fixed-point-q}.
Since $G$ is differentiable, we see that
\[
\frac{\xi''(q)}{2}(\E (\partial_x v)^2(q,X_q)-q)=0.
\]
This yields \prettyref{eq:fixed-point-q} for $q\neq0$, and for $q=0$ if $\xi''(0)\neq 0$. 
To see this for the point $q=0$, it suffices to show that
\begin{equation}\label{eq:support-lower-bound}
(\partial_x v(0,0))^2 \leq \inf\supp(\mu).
\end{equation}
To this end, let $q_0=\inf\supp(\mu)$. Observe that by \prettyref{eq:G-minimizing},
\[
G(q_0)\leq G(q_0+\eps)
\]
for $\eps>0$ sufficiently small. Averaging this inequality
and using the definition of $G$, we see that there is some 
$t\in(q_0,q_0+\eps)$ such that
\[
\E (\partial_x v(t,X_t))^2 \leq t.
\]
As observed in \prettyref{eq:max-princ}, $\partial_x^2 u>0$. So by It\^o's isometry, $t\mapsto\E (\partial_xv)^2(t,X_t)$ is strictly increasing.
Thus 
\[
(\partial_x {v})^2(0,0)=\E(\partial_x {v})^2(0,{\hat X_0})\leq t\leq q_0+\eps.
\]
Sending $\eps\to0$ yields \prettyref{eq:support-lower-bound}.
The remaining cases can be proved in an identical fashion.
\end{proof}

\subsection{$\Gamma$-convergence of the local free energy}\label{sec:gamma-conv-local-fe}
We now turn to proving the $\Gamma$-convergence of the local free energy functional.
We begin by recalling the notion of sequential $\Gamma$-convergence. 

\begin{defn}
Let $X$ be a topological space. We say that a sequence of functionals $F_n : X \to [- \infty, \infty]$ \textit{sequentially $\Gamma$-converges} to $F : X \to [-\infty, \infty]$ if 
\begin{enumerate}
\item The $\Gamma-\liminf$ inequality holds: For every $x$ and sequence $x_n \to x$, $$\liminf_{n \to \infty} F_n(x_n) \geq F(x).$$ 
\item The $\Gamma-\limsup$ inequality holds: For every $x$, there exists a sequence $x_n \to x$ such that $$\limsup_{n \to \infty} F_n(x_n) \leq F(x).$$
\end{enumerate}
For a sequence of functionals $F_{\beta}$ indexed by a real parameter $\beta$, we say that $F_{\beta}$ sequentially $\Gamma$-converges to $F$ if for any sequence $\beta_n \to \infty$, the sequence $F_{\beta_n}$ sequentially $\Gamma$ converges to $F$. 
\end{defn}

\begin{cor}\label{cor:gamma-conv-main-thm}
For every $T\in[\dlower,\dupper]$, we have that 
\[
{\cP_{\beta, T}\stackrel{\Gamma}{\to}\cP_T.}
\]
\end{cor}

\begin{proof}
Recall from \eqref{eq:pfunc-def} that we may write $\cP_{\beta,T}$ in the form
\[
\cP_{\beta,T}(\nu,\lambda)=\cF_{\beta,T}(\nu,\lambda; 0,0) +\ell_1(\lambda)+\ell_2(\nu),
\]
where $\ell_i$ are both linear functionals that do not vary in $\beta$. 
For $(t,x) \in [0,T]\times \mathbb{R}$, set 
\[
\cF_T(\nu,\lambda;t,x) = u_{\nu,\lambda}(t,x).
\]
For $(t,x) \in [0,T]\times \mathbb{R}$, recall the functional $\mathcal{F}_{\beta,T}(\cdot, \cdot; t,x)$ from \eqref{eq:finite_temp_functional}, and note that Theorem \ref{thm:annealed-soln-thm} immediately implies that $\mathcal{F}_{\beta,T}(\cdot, \cdot; 0,0) \stackrel{\Gamma}{\to} \mathcal{F}_{T}(\cdot, \cdot; 0,0)$. The desired conclusion follows immediately using \eqref{eq:local-zero-temp-pfunc} and the stability of $\Gamma$-convergence under continuous perturbations \cite{Braides02}.
\end{proof}

Our interest in the $\Gamma$-convergence of these functions
is of course to understand convergence of minima and minimizers.
To this end, we need a precompactness theorem for such minimizers.

\begin{thm}\label{thm:conv-minimizers-thm}
If $T\in(\dlower,\dupper)$, then $\cP_{\beta,T}$ has a unique~ sequence of minimizers
$(\nu_{\beta},\lambda_{\beta})$ which is precompact. Furthermore
any limit point of this sequence converges to a minimizer of $\cP_T(\nu,\lambda)$.
If $T=\dlower$ or $\dupper$, then the family $\nu_{\beta}$ is precompact, we
may take $\lambda=0$,
and any limit point of this
sequence is such that $(\nu,\lambda)$ is a minimizing pair.
\end{thm}

Let us now turn to the precompactness theorem \prettyref{thm:conv-minimizers-thm}. 
Before we begin the proof, 
we need the following theorem regarding the compactness of $\lambda$. 

\begin{lem}
There is a $\beta_0(\Sigma)$ such that if $\beta\geq\beta_0$, 
$T\in(\dlower,\dupper)$, and $\cP_{\beta,T}(\lambda,\nu)\leq M$, then 
\begin{equation}\label{eq:coercive-lambda-bound}
\abs{\lambda}\leq\frac{M+1}{\min\{\dupper-T,T-\dlower\}}.
\end{equation}
\end{lem}
\begin{proof}
Observing that the last term in \eqref{eq:pfunc-def} is negative, 
\[M\geq u_{\nu,\lambda}(0,0)-\lambda T.\] 
 In this case, 
\prettyref{lem:mgf-bound} implies that 
\[
M \geq \frac{1}{\beta}\psi(\lambda\beta)-\lambda T.
\] 
Observe that there is a $c>0$ that depends only on $\Sigma$ such that
\[
\frac{1}{\beta}\psi(\beta\lambda)\geq \max\{\lambda \dlower,\lambda \dupper\} -\frac{c}{\beta}.
\]
Taking $\beta_0=c^{-1}$ and re-arranging yields the result.
\end{proof}

We are now in the position to prove the precompactness theorem.
\begin{proof}[\textbf{\emph{Proof of \prettyref{thm:conv-minimizers-thm}}}]

That there is a unique minimizing pair for finite $\beta$ is proved in \prettyref{lem:optimality-finite-beta}.
We begin by studying the precompactness of this sequence.

Let us first show that the collection of minimizing $\nu_\beta$ are precompact.
 To this end, observe that,
by \prettyref{eq:loc-fe-conv},
 $F(\beta)$ is the point-wise limit of free energies. 
As functions, these are convex and smooth in $\beta$.
Furthermore, they satisfy
 \[
 \partial_\beta F(\beta;T) = \lim_{N\to\infty} F'_N(\beta;T)=\lim_{N\to\infty} \frac{1}{N}\E \langle H_N(\sigma)\rangle \leq C(\xi).
 \]
Here $\langle\cdot\rangle$ denotes integration with respect to the Gibbs measure,
\[
\eta(\{\sigma\})\propto e^{\beta H_N}
\]
and we use that 
the expect normalized maximum of $H_N$ is bounded by a function of $\xi$ alone \cite[Theorem 2.5]{BLM}. 

 By \prettyref{thm:beta-deriv-overlap} and Fubini's theorem, it then follows that 
 \[
 \int\xi'(t)\beta\mu(t)dt\leq C(\xi)
 \]
 By the Harris-FKG inequality, this implies that the total variation norm of $\nu_\beta$ satisfies $\norm{\nu_\beta}\leq C'(\xi)$. Thus the minimizing $\nu_\beta$
 are pre-compact.
 
We now study the pre-compactness of $\lambda_\beta$. 
Suppose first that $T\in(\dlower,\dupper)$. In this case,  there is some $M$ such that  eventually
 \[
 \cP_{\beta,T}(\nu_\beta,\lambda_\beta)\leq M.
 \]
 Similarly by the above estimate, we may assume that $\norm{\nu_\beta}\leq C$. 
 The result then follows by  \prettyref{eq:coercive-lambda-bound}.
The case $T=\dlower$ and $T=\dupper$ are obvious. 
 \end{proof}
 
\subsection{Variational representation for the Ground State Energy}
With the above in hand, the proof of  Theorem \ref{thm:variational_rep} is essentially immediate.  

\begin{proof}[\textbf{\emph{Proof of Theorem \ref{thm:variational_rep}}}]   \prettyref{cor:gamma-conv-main-thm} establishes $\Gamma$-convergence of the functional $\mathcal{P}_{\beta,T}$ to $\mathcal{P}_{T}$. 
Furthermore, the minimizers of $\mathcal{P}_{\beta,T}$ are pre-compact by Theorem \ref{thm:conv-minimizers-thm}. Thus by the fundamental theorem of $\Gamma$-convergence, the minima converge, i.e., 
\begin{align}
\inf_{\nu \in \mathcal{A}_T, \lambda \in \mathbb{R} } \mathcal{P}_{\beta,T}(\nu, \lambda)  \to \inf_{\nu \in \mathcal{A}_T, \lambda \in \mathbb{R}} \mathcal{P}_T(\nu,\lambda)\label{eq:gamma_minima_conv}
\end{align}
as $\beta \to \infty$. Lemma \ref{lemma:ground_state_approx} implies 
\begin{align}
\frac{1}{\beta} F_N(\beta, \xi; A_N) - \frac{\log |\Sigma|}{\beta} \leq  GS_N(A_N) \leq \frac{1}{\beta} F_N(\beta, \xi; A_N) + \frac{\log |\Sigma| }{\beta}. \nonumber 
\end{align}
We let $N\to \infty$, followed by $\beta \to \infty$ and use \eqref{eq:gamma_minima_conv} to derive that 
\begin{align}
\lim_{N \to \infty}  GS_N(A_N) = \inf_{\nu \in \mathcal{A}_T, \lambda \in \mathbb{R} } \mathcal{P}_T(\nu,\lambda) =: E(\xi;T). \nonumber 
\end{align}
Another application of Lemma \ref{eq:gamma_minima_conv} implies that 
\begin{align}
|\frac{1}{\beta} F(\beta, \xi; T) - E(\xi; T) | \leq \frac{\log |\Sigma|}{\beta}. \nonumber
\end{align}
Thus the family $\{\frac{1}{\beta}  F(\beta , \xi; \cdot) : \beta >0 \}$ is uniformly convergent, and thus the supremum converges, 
\begin{align}
\sup_{T} \frac{1}{\beta} F(\beta, \xi ; T) \to \sup_{T} E(\xi ; T) \nonumber 
\end{align}
when $\beta \to \infty$. Finally, using Lemma \ref{eq:gamma_minima_conv}, we have
\begin{align}
\liminf_{N \to \infty} GS_N(\Sigma^N) &\geq \liminf_{N \to \infty} F_{N}(\beta, \xi ) - \frac{\log |\Sigma|}{\beta}\nonumber \\
&\geq \liminf_{N \to \infty} F_N(\beta, \xi; T) - \frac{\log |\Sigma|}{\beta} = F(\beta,\xi; T)  - \frac{\log |\Sigma|}{\beta}. \nonumber
\end{align}
We let $\beta \to \infty$ and take supremum over $T$ to derive the lower bound 
\begin{align}
\liminf_{N \to \infty} GS_N(\Sigma^N) \geq \sup_{T} E(\xi; T). \nonumber
\end{align}
To derive the upper bound, we observe that  
\begin{align}
\limsup_{N \to \infty} GS_N(\Sigma^N) \leq  \limsup_{N \to \infty} F_N(\beta, \xi) + \frac{\log |\Sigma|}{\beta} = F(\beta, \xi) + \frac{\log |\Sigma|}{\beta}.  \nonumber
\end{align}
It is easy to see that $F(\beta, \xi) = \sup_{T} F(\beta, \xi; T)$ and then we let $\beta \to \infty$ to obtain 
\[
\limsup_{N \to \infty} GS_N(\Sigma^N) \leq \sup_{T} E(\xi, T). \nonumber 
\] \end{proof}

\section{Analysis of the Zero temperature problem}\label{sec:analysis-zero-temp}

In this section, we briefly turn to calculating the first variation of the functional $\cP_T$ from \eqref{eq:local-zero-temp-pfunc}  

\begin{lem}
\label{lemma:ground_state_derivative}
Fix $\nu_{0},\nu_{1}\in\cA$ with $\nu_{1}\left(\{T\}\right)=\nu_{0}\left(\{T\}\right)$ and $\lambda$.
Let $\nu_{\theta}=(1-\theta)\nu_{0}+\theta\nu_{1}$. We have that
\[
\frac{d}{d\theta}\vert_{\theta=0}\cP_T(\nu_{\theta})=\frac{1}{2}\int_{0}^{T}\xi''(t)\left(\E u_{x}(t,X_{t})^{2}-t\right)d(\nu_{1}-\nu_{0}).
\]
\end{lem}
\begin{proof}
Let $m_{\theta}=(1-\theta)m_{0}+\theta m_{1}$. Let $X_{T}^{\alpha,m_{\theta}}$
be the process in \prettyref{lem:dpp} corresponding to $m_{\theta}$ with initial
data $x=0$. Consider the auxiliary function $\Xi:\cB_{T}\times[0,1]\to\R$
\begin{align*}
\Xi(\alpha,\theta) & =\E\left[\psi(X_{T}^{\alpha,m_{\theta}},\lambda+\frac{\xi''(T)}{2}c)-\frac{1}{2}\int_{0}^{T}\xi''(s)m_{\theta}(s)\alpha_{s}^{2}ds\right].
\end{align*}
Since $\psi$ is continuous, it is clear that $\Xi$ is jointly continuous
in $(\alpha,\theta)$. 
Consider the function 
\[
\psi(x,y)=\max_{\spin\in\Sigma}\Big[ \spin x+\spin^{2}y \Big].
\]
Recall from \prettyref{lem:psi-diff} that $\psi$ is a.e. differentiable
with a.e. continuous derivative which satisfies
\[
\partial_{x}\psi(x,y)=\spin_{*},
\]
where $\spin_{*}$ is such that $\psi(x,y)=\spin_{*}x+\spin_{*}^{2}y$. 
Thus $\partial_{\theta}\Xi(\alpha,\theta)$ is jointly continuous
in the pair $\theta,\alpha$. By \prettyref{lem:dpp}, if we let
\[
\alpha_{*}(\theta)=u_{x}^{\theta}(s,X_{s}^{\theta}),
\]
then these achieve optimality in \eqref{eq:dpp-zero-temp}. Furthermore, the map $\theta\mapsto\alpha_*(\theta)$
is continuous by \prettyref{lem:continuity}. Thus by \prettyref{lem:wei-kuo}, we have that 
\[
\frac{d}{d\theta}\vert_{\theta=0}\cP_T(\nu_{\theta})=\frac{1}{2}\int_{0}^{T}\xi''(t)\left(\E u_{x}^{2}-t\right)(m_1(t)-m_0)dt,
\]
as desired.
\end{proof}

Our next result is a convexity property of the zero temperature functional $\cP_T$, and should be compared to Lemma \ref{lem:strict-convexity-finite-beta}. 
\begin{lem}
\label{lemma:ground_state_convexity}
For $\lambda \in \mathbb{R}$ and $\nu \in \mathcal{A}$, $\de\nu(t) = m(s) \de t + c \delta_{T}$, 
the ground state Parisi functional $\cP_T$ is convex in $(m,c,\lambda)$. 
\end{lem}
Combining Lemmas \ref{lemma:ground_state_derivative} and \ref{lemma:ground_state_convexity} immediately yields the following corollary. 
\begin{cor}
\label{cor:derivative_convexity}
Let $\lambda_*, \nu_*$ be any minimizer of $\cP_T$. Set $\de \nu_*(s) = m_*(s) \de s + c_* \delta_T$. Consider any measure $\de \nu_1(s) = m_1(s) \de s + c_* \delta_T$ and define the path $\nu_{\theta} = \theta \nu_1 + (1- \theta) \nu_*$. Then we have, 
\begin{align}
\frac{\de}{\de \theta}^{+} \cP_T(\lambda_*, \nu_{\theta}) \Big|_{\theta=0} \geq 0. \nonumber
\end{align}
\end{cor}

 \section{Proof of Theorem \ref{thm:gamma_conv_application}}
 \label{subsec:gamma_conv_app}
 We prove Theorem \ref{thm:gamma_conv_application} in this section. 
Recall the interpolating free energy $F_N(v,\beta;\alpha)$ from  \eqref{eq:fn_defn}. Further, recall that by Panchenko's theorem \cite{panchenko2005generalized}, see \eqref{eq:loc-fe-conv},
\[
F_N(v,\beta ; \alpha )\to F(v,\beta; \alpha),
\]
where $F(v,\beta; \alpha)$ is the local free energy \eqref{eq:loc-fe} corresponding to $\xi(t)=2 t^2$,
$\Sigma=\Sigma(v,\alpha)$ and 
$T=T(v,\alpha)$.{\footnote{We note here that we may take $\eps_N =0$ in \eqref{eq:loc-fe-conv}. The ``fattening'' by $\eps_N$ was necessary in \cite{panchenko2005generalized}, only because they needed to work with self-overlaps which were possibly un-realizable, e.g., irrational. In our setting, however, we are implicitly in the regime where the set with self-overlap $T$ is non-empty infinitely often in $N$. }}

Our proof of Theorem \ref{thm:gamma_conv_application} will crucially use the ground state energy functional \eqref{eq:local-zero-temp-pfunc}.  
We adapt the ground state energy functional  to this setting for the convenience of the reader. Fix $v \in [0,1]$ and define the functional $\mathcal{P}(\cdot, \cdot ;v) : \mathbb{R} \times \mathcal{A}_T \to \mathbb{R}$ such that 
\begin{align}
\mathcal{P}(\lambda, \nu; v) = - \lambda T + \tilde{u}_{\lambda, \nu}(0,0) - 2 \int_0^{T} s \de \nu(s), \nonumber 
\end{align}
where we set $T:= T(v)$ as in \eqref{eq:t_defn}, and where for  $\de \nu(s) = m(s) \de s + c \delta_T$, $\tilde{u}_{\lambda,\nu}$ solves 
\begin{equation}\label{eq:ground_state_unbalanced}
\begin{cases}
\partial_t \tilde{u}_{\lambda, \nu} + 2 \Big( \Delta \tilde{u}_{\lambda,\nu} + m(t) (\partial_x \tilde{u}_{\lambda,\nu} )^2 \Big)=0, &(t,x)\in [0,T]\times\R \\
\tilde{u}_{\lambda,\nu}(T, x) = \tilde f(x,\lambda,c). 
\end{cases}
\end{equation}
where 
\[
\tilde f(x,\lambda,c) = \Big| x - 2 \Big( \lambda + 2c \Big)\sqrt{v} (2 \alpha -1) \Big| - \sqrt{v} (2 \alpha -1 ) x + \Big( \lambda + 2c\Big) \Big( 1 + v(2 \alpha -1)^2 \Big).
\]

 \noindent

Recall the local Parisi functional $\mathcal{P}_{\beta,T}(\nu, \lambda)$ \eqref{eq:loc-fe}, and the pre-compactness of its minimizers, as established in Theorem \ref{thm:conv-minimizers-thm}. 
Next, we will establish the following crucial property about the minimizers. 
\begin{lem}
\label{thm:ground_state_non-zero}
Fix any $\alpha \in (0, \frac{1}{2})$ and $v \in (0,1)$. Let $(\nu_*, \lambda_*)$ be any limit point of the minimizers $(\nu_{\beta}, \lambda_{\beta})$. Then $\nu_* \neq 0$. 
\end{lem}
We defer the proof of this result~ to the end of the section.
Given these results, we are now in a position to establish Theorem \ref{thm:gamma_conv_application}.

\begin{proof}[\textbf{\emph{Proof of Theorem \ref{thm:gamma_conv_application}}}]
 By  Gaussian integration of parts for Gibbs measures \cite[(3.98)]{Pan}, we observe that 
 \begin{align}
 \partial_{\beta} F_N(v,  \beta ; \alpha) =2 \beta \E[T^2 - \langle R_{12}^2 \rangle_v]. \nonumber
 \end{align}
 We note that for any fixed $v, \alpha$, $F_N( v, \beta; \alpha )$ is convex in $\beta$ and converges to $F(v, \beta; \alpha )$, which is differentiable in $\beta$ by Theorem \ref{thm:beta-deriv-overlap}. Thus by Griffith's lemma for convex functions, 
 \begin{align}
 \partial_{\beta} F_N(v, \beta; \alpha) \to \partial_{\beta} F(v, \beta; \alpha). \nonumber
 \end{align}
 Using Theorem \ref{thm:beta-deriv-overlap}, we have 
 \begin{align}
 \partial_{\beta} F(v, \beta; \alpha) =2 \beta \int (T^2 - x^2 ) \de \mu_{\beta}(x), \nonumber 
 \end{align}
 where $\mu_{\beta}$ is the minimizer of the local free energy functional $\mathcal{P}_{\beta,T}$. 
 The minimizers of $\mathcal{P}(\cdot, \cdot; v)$ are functions of $v$, but for ease of notation, we will keep this dependence implicit.

We set $\de \nu_{\beta}(t) = \beta \mu_{\beta}([0,t]) \de t$, where $\mu_{\beta}$ is unique minimizer of $\mathcal{P}_{\beta,T}$. Let $(\nu_*, \lambda_*)$ be any limit point of $(\nu_{\beta}, \lambda_{\beta})$.  Recall that by Theorem \ref{thm:conv-minimizers-thm}, such a limit point exists. 
Using \prettyref{lem:convergence} and  Lemma \ref{thm:ground_state_non-zero}, we have,  that for any subsequence converging to this limit point, 
 \begin{align}
 \lim_{k \to \infty} \beta_{k} \int (T^2 - x^2) \de \mu_{\beta_k}(x)  =  \int 2 x \de\nu_*(x)>0.   \nonumber 
 \end{align}
This observation implies  
 \begin{align}
 \liminf_{\beta \to \infty} \beta \int (T^2 - x^2) \de\mu_{\beta}(x)>0 . \nonumber 
 \end{align}
 Setting 
 \begin{align}
 C_1(\alpha) = \liminf_{\beta \to \infty} \beta \int (T^2 - x^2) \de\mu_{\beta}(x) >0 \nonumber 
 \end{align}
 gives us the desired constant, and completes the proof. 
 \end{proof}

It remains to prove Lemma \ref{thm:ground_state_non-zero}. We outline this in the rest of the section.

\begin{proof}[\textbf{\emph{Proof of Lemma \ref{thm:ground_state_non-zero}}}]
 Fix $v \in (0,1)$ and assume for the sake of contradiction that $( 0,\lambda_*)$  is a limit point of $(\nu_{\beta},\lambda_{\beta})$. By Theorem \ref{thm:conv-minimizers-thm}, $(\lambda_*, 0)$ is a minimizer of $\mathcal{P}(\cdot, \cdot; v)$ and $\lambda_*$ is finite. 
For any probability measure $\mu$ on $[0,T]$, we can construct the path on measures $\nu_{\theta} =  (1- \theta) \nu_1$, where we set $\de \nu_1(s) = \mu([0,s]) \de s$. In this case, if we apply Fubini's theorem to  Lemma \ref{lemma:ground_state_derivative}, we obtain
\begin{align}
\frac{\de}{\de \theta}^{+} \mathcal{P}( \nu_{\theta}, \lambda; v) \Big|_{\theta=0}= \int G_v  \de \mu  , \nonumber
\end{align} 
where $G_v$ is defined as
\[
G_{v}(t)=\int_t^T\xi''(s)(\E (\partial_x \tilde{u})^2(s,X_s) - s)ds,
\]
where $\tilde{u}$ is the solution to \eqref{eq:ground_state_unbalanced} corresponding to $0$. 
As $0$ as assumed to be a minimizer, an application of Corollary \ref{cor:derivative_convexity} implies 
$\int G_v \de \mu \geq 0$. Further, $\mu$ is an arbitrary probability measure on $[0,T]$, 
and thus $G_v(s ) \geq 0$ on $[0,T]$. We will establish that $G_v(T) = G'_v(T) = 0$ while 
$\lim_{t \uparrow T} G''_v(t) = \infty$. Thus $G_v(t)$ is negative for $t$ sufficiently close to $T$, and this yields a contradiction.

To this end, we note that the definition of $G_v$ immediately implies that $G_v(T) =0$. We \text{will} next establish that 
\begin{align}
\E[(\partial_x\tilde{u})^2(T, B_{4T})] =T, \nonumber
\end{align}
which implies $G'_v(T) =0$. To this end, note that the weak derivatives of $\tilde{u}$ satisfy the relation
\begin{align} 
(\partial_x \tilde{u})^2 (T ,x) = \partial_{\lambda} \tilde{u}(T,x). \nonumber
\end{align}
Now, note that $\lambda_{\beta}$ are pre-compact, and thus bounded, implying that $\lambda_*$ is finite. 
Differentiating the functional in $\lambda$ as in \prettyref{lem:optimality-finite-beta}, and setting this to zero,  we obtain 
\begin{align}
T = \partial_{\lambda_*} \tilde{u}(0,0) = \E[\partial_{\lambda_*} \tilde{u}(T, B_{4T}) ] = \E[(\partial_x \tilde{u})^2(T, B_{4T})], \nonumber
\end{align}
where the second equality follows from the observation that $\partial_{\lambda_*}\tilde{u}$ satisfies the heat equation with boundary data $\partial_{\lambda_*}\tilde{u}(T, x)$. 

Finally, we prove that $\lim_{t \uparrow T} G_v''(t) = \infty$. We have, for $t < T$, 
\begin{align}
G_v''(t) = 4 \Big( \frac{\de}{\de t }\E[( \tilde{u}_x)^2(t, B_{4t})] - 1 \Big). \nonumber 
\end{align}
Using It\^o Lemma's we immediately obtain that for such $t$,
\begin{align*}
\frac{\de}{\de t} \E[(\tilde{u}_x)^2(t, B_{4t})] =4 \E[ (\tilde{u}_{xx})^2(t,B_{4t})].  
\end{align*}
Note that $\tilde{u}_{xx}$ solves the heat equation 
\[
\begin{cases}
( \partial_t  +  \frac{4}{2} \Delta )\tilde{u}_{xx} = 0  & \\
\tilde{u}_{xx}(T, \cdot ) =2  \delta_{2 \lambda_* a(v)}(\cdot) &
\end{cases}
\]
in the sense of distributions, where $a(v) = \sqrt{v}(2\alpha -1)$. In particular, by a standard argument \cite{JohnPDE} (or an explicit computation) we have,
\begin{align}
\tilde{u}_{xx}(t,x) = \frac{2}{\sqrt{8\pi (T-t) }} \exp\Big( - \frac{(x- 2 \lambda_* a(v) )^2}{8(T-t)} \Big), \nonumber 
\end{align}
for $t < T$. Finally, this immediately implies 
\begin{align}
\E[(\tilde{u}_{xx})^2(t, B_{4t})  ] = \frac{1}{2\pi (T-t)} \E\Big[\exp\left( - \frac{(B_{4t} - 2 \lambda_* a(v))^2}{4(T-t)}\right) \Big] = \frac{1}{2 \pi \sqrt{T^2 - t^2} } \exp\Big(- \frac{ ( \lambda_* a(v) )^2 }{T+t}\Big). \nonumber
\end{align}
We let $t \uparrow T$ to complete the proof. 
\end{proof}

\section{A conjecture regarding the sharp constant in \prettyref{thm:cut_difference}}\label{sec:conj1}
{In this section, we record the conjecture regarding the sharp constant in Theorem \ref{thm:cut_difference}. To this end}
, define
 \begin{align}
 \Sigma(v,\alpha) &= \{ 1- \sqrt{v} ( 2 \alpha -1) , - 1 - \sqrt{v} (2\alpha -1) \}, \nonumber\\
 T(v,\alpha) &= \alpha (1 - \sqrt{v} (2 \alpha -1 ))^2 + (1 - \alpha) (1 + \sqrt{v}(2 \alpha -1 ))^2 , \nonumber
 \end{align}
and then we 
have the following.

\begin{conjecture}
We have that
\[
\lim_{d\to\infty}\lim_{N\to\infty}\frac{\Mcut_{\alpha}(G(N,\frac{d}{N}))-\Mcut_{\alpha}(\Rg{d})}{\sqrt{d}N}=K(\alpha),
\]
where
\[
K(\alpha)=2(2\alpha-1)^2\int_0^1\frac{ (1 -\sqrt{v})}{\sqrt{v}}\nu_v([0,T(v,\alpha)])dv.
\]
and where $\nu_v$ is a minimizer of $\cP_{T(v,\alpha)}$.
\end{conjecture}

The motivation for this conjecture is as follows.
By \prettyref{lem:int_by_parts}
one formally expects that
\begin{align*}
\lim_{d\to\infty}\lim_{N\to\infty}&\frac{\Mcut_{\alpha}(G(N,\frac{d}{N}))-\Mcut_{\alpha}(\Rg{d})}{\sqrt{d}N}\\
&\qquad\qquad=\lim_{\beta\to\infty}\lim_{N\to\infty} 2(2\alpha-1)^2\int \frac{1-\sqrt{v}}{\sqrt{v}}\beta\E\gibbs{T(v)-R_{12}}_v dv,
\end{align*}
In the physics literature, a basic tenet of the replica symmetry breaking method is  \cite{parisi1983order} that in generic situations
we have the correspondence
\[
\lim_{N\to\infty}\E\gibbs{T-R_{12}}=\beta\int (T-x)d\mu_\beta(x),
\]
where $\mu_\beta$ is such that $\beta\mu_\beta dt$ the minimizer of $\cP_{\beta,T}$ from \eqref{eq:pfunc-def}. 
The conjecture then comes from combining this correspondence with \prettyref{thm:conv-minimizers-thm}. 
The question as to when this correspondence holds is a major open problem in the mathematical study
of mean field spin glasses. For references in this direction see \cite{TalBK11Vol2,Pan,Panch08}.

\appendices
\section{Basic Properties of the Parisi PDE}\label{app:pde}
In this section, we briefly review the properties of solutions to Parisi-type PDEs.

\begin{lem}\label{lem:ppde-reg-zero-temp}
There is a unique weak solution to \prettyref{eq:zero-temp-pde}, which satisfies 
\begin{itemize}
\item $\partial_x u\in L_{t}^{\infty}L_{x}^{\infty}$ with $\norm{\partial_x u}_{L_{t,x}^{\infty}}\leq\norm{\partial_x f}_{L_{t,x}^{\infty}}\leq C(\Sigma)$
\item For any $T_{0}<T$, $u$ is continuous in space time, with smooth
bounded spatial derivatives satisfying 
\[
\norm{\partial_{x}^{n}u}_{L^{\infty}([0,T_0]\times \mathbb{R} )} \leq C(T_{0},T,\Sigma)
\]
and is weakly differentiable in time with $\norm{\partial_{t}\partial_{x}^{n}u}_{{L^{\infty}([0,T_0]\times \mathbb{R})}}\leq C(T_{0},T,\Sigma)$.
\end{itemize}
Furthermore, if $u,v$ are two solutions corresponding to $\mu$ and
$\nu$ respectively where $\mu,\nu\in\cA$ are of the form 
\begin{align*}
\mu & =m_{1}dt+c\delta_{1}\\
\nu & =m_{2}dt+c\delta_{1}
\end{align*}
then
\[
\norm{u-v}_{L_{t,x}^{\infty}}\leq C(\xi)\norm{m_{1}-m_{2}}_{L^{1}}
\]
and we have the parabolic comparison principle:
 if $m_1\leq m_2$ pointwise for all $t$ then $u\leq v$. Finally, the same results hold for the weak solutions $u$ to \eqref{eq:ppde-finite-beta}.
\end{lem}
\begin{proof}
This is a standard argument, see, e.g., \cite{JagTobPD15,auffinger2015properties}. For the reader's convenience
we briefly sketch the main points. We begin first with the existence
for a dense class of $\nu$'s. Assume that $\nu=m(t)dt+c\delta_{1}$.
The existence for $m$ which is a bounded step function with finitely
many jumps, can be seen by an application of the Cole-Hopf transformation.
That the derivative is bounded in space for such solutions can be
seen either by explicit differentiation or the maximum principle.
This yields
\[
\norm{\partial_x u}_{L_{t,x}^{\infty}}\leq\norm{\partial_{x}f}_{L^\infty_x}.
\]
By \prettyref{lem:psi-diff},  $f$ is differentiable in $x$ Lebesgue a.e. 
and  $\partial_{x}f \in L^{\infty}$. Furthemore,
\[
\partial_{x}f=\textrm{argmax}_{\epsilon\in\Sigma} \{ x\epsilon+(\lambda+\frac{c}{2}\xi''(T))\epsilon^{2}\}
\]
a.e. which is bounded by a constant that depends at most on $\Sigma$.
Observe furthermore, that the regularity claims in this setting can
be seen by explicit differentiation. 

We now prove the Lipschitz estimate in $m$ for $m$ as above and the corresponding comparison principle. This
follows by the same argument as in \cite[Lemma 14]{jagannath2015dynamic}. Indeed,
if $w=u-v$ then $w$ solves 
\begin{align*}
w_{t}+\frac{\xi''}{2}\left(w_{xx}+m_{1}(u_{x}+v_{x})w_{x}+(m_{1}-m_{2})v_{x}^{2}\right) & =0
\end{align*}
with initial data $w(T,x)=0$. Since $u_{x},v_{x}$ are uniformly
lipschitz on any subinterval of the form $[0,T_{0}]\subset[0,T)$,
and $m_{i}(t)$ are both uniformly bounded on such sub intervals,
we see that we may solve the SDE 
\[
dX_{t}=\xi''m_{1}\frac{1}{2}\left(u_{x}+v_{x}\right)(t,X_{t})+\sqrt{\xi''}dW_{t},
\]
where $W_{t}$ is a standard Brownian motion. Observe that $w$ has
the same regularity as $u$ and $v$. In particular, by the smoothing
property of the heat equation, we have that on $[0,T_{0}]\times\R$
$w,w_{x},w_{xx}\in C_{b}([0,T_{0}],\R)$, and $w$ is weakly differentiable
in time with $w_{t}\in L^{\infty}$ which is Lipschitz in $x$ uniformly
in $t$. Thus we may apply It\^o's lemma (see, e.g., \cite[Proposition 22]{jagannath2015dynamic}) to obtain 
\[
w(t,x)=\E_{X_{t}=x}\left(\int_{t}^{T_{0}}\frac{1}{2}\xi''(s)\left(m_{1}-m_{2}\right)v_{x}^{2}\right)\leq C(\Sigma,\xi)\norm{m_{1}-m_{2}}_{L^{1}}.
\]
Similarly, sending $t\to0$ and $T_{0}\to T$, yields the desired conclusions. 

We now show the existence for general $m$. By an extension argument,
if $m_{n}\to m$ in $L^{1}$, then $u_{n}\to u$ for some function
$u$. To see that $u$ is a weak solution observe that it suffices
to show that 
\[
m_{n}\left(\partial_{x}u_{n}\right)^{2}\to m(\partial_{x}u)^{2}
\]
 in the sense of distributions. This follows since $\partial_{x}u_{n}$
are uniformly bounded. To prove uniqueness, observe that by a similar
Duhamel's principle argument to \cite[Lemma 13]{jagannath2015dynamic} using the modified
heat kernel estimates from \cite[Appendix 1]{BAJag17}, we have that 
\[
\norm{\partial_{x}w}_{{L^{\infty}([t,T]\times \mathbb{R})}}\leq\norm{u_{x}+v_{x}}_{{L^{\infty}_{t,x}}}\int_{t}^{T}\norm{\partial_{x}w}_{{L^{\infty}([s,T] \times \mathbb{R})}}\frac{C(\Sigma,\xi)}{\sqrt{s-t}}m(s)ds.
\]
Since $m(t)$ is monotone and blows up at most at $T$, the integrand
is integrable, so that by Gronwall's inequality, $\norm{\partial_{x}w}=0$.
Thus we have the existence and uniqueness of $u$ and the regularity
of $u_{x}$. The regularity of the higher spatial derivatives follows
by the smoothing property of the heat semigroup, and the regularity
in time follows from rearranging \prettyref{eq:zero-temp-pde} and using the fact that
on $T_{0}<T$, $m$ is bounded. 
\end{proof}

\begin{lem}\label{lem:continuity}
The map $\nu\mapsto u_{x}$ is continuous in the topology of pointwise
convergence. In particular, if $\nu_{n}\to\nu$ weak-{*}, then $u^{n}\to u$
and $\partial_{x}u^{n}\to\partial_{x}u$ uniformly on compacta.
\end{lem}
\begin{proof}
By the same Duhammel's principle and parabolic Bootstrapping argument
as in  \cite[Lemma 9,10]{chl2016energylandscape}, there are a continous function
$\{F_{i}(x,y)\}_{i\in[3]}$ such that for every $T_{0}<T$, 
\begin{align*}
\norm{\partial_{x}^{i}u}_{L^{\infty}([0,T_{0}]\times\R)} & \leq F_{i}(\nu(T_{0}),\nu(\{1\})).
\end{align*}
Since the maps $\nu\mapsto\nu([0,t])$ and $\nu\mapsto\nu\left(\{1\}\right)$
are upper semi-continuous in the weak-{*} topology on $\cA$, we see
that if $\nu_{n}\to\nu$, weak-{*} the family $\{u^{n}\}$ and $\{\partial_{x}u^{n}\}$
are uniformly lipschitz in space. Furthermore, since they weakly solve
the Parisi PDE and the (spatially) differentiated form of this equation,
we see that they are also uniformly lipschitz in time. Thus the families
are both equicontinuous. Thus $u^{n}\to u$ and $\partial_{x}u^{n}\to\partial_{x}u$
uniformly on compacts by the Arzela-Aiscoli theorem. 
\end{proof}

\section{The weak-* topology on $\cA_T$}\label{app:appendix-at}
In the preceding section, we frequently work with the space $\cA_T$ equipped with
the weak-* topology. We provide here certain basic properties
of this space. 

\begin{lem}\label{lem:convergence}
Suppose that $\nu_\beta\to \nu$ with $\nu_\beta = m_\beta(t)dt+c_\beta\delta_T$ and $\nu=m(t)dt+c\delta_T$. We then have the following.
\begin{enumerate}
\item[(i)] The measure $dm_\beta$ converges vaguely to $dm$. That is, for every $t<T$, 
that is a continuity point of $m$,
\[
m_\beta(t)\to m(t).
\]
\item[(ii)] Let $q_\beta\to1$  be such that $m_\beta(q_\beta^-)\to m(1^-)$, where the $-$ denote the left limit,
and suppose that $m(1^-)<\infty$. Then
\[
c = \lim \fint_{[q_\beta,T]}\beta(T-t)dm_\beta. 
\] 
\item[(iii)] For any $f\in C^1([0,T])$,
\[
\lim \beta \int f(T)-f(t)dm_{\beta} = \int f' d\nu.
\]
\item[(iv)] For any bounded Borel measurable $\psi$ with $\lim_{s\uparrow T}\psi(s)=\psi(T)$,
\[
\int \psi(s)d\nu_\beta(s)\to\int \psi(s)d\nu(s).
\]
\end{enumerate}
\end{lem}
\begin{proof}
The first three points were proved in \cite{jagTob17}. It remains to prove the last point.
Let $\psi$ be as above.  Without loss of generality, assume that $\norm{\psi}_\infty \leq1$.
Observe that by (1), since $\nu_{\beta} \to \nu$  weak-*, $m_{\beta}(t) \to m(t)$ Lebesgue
almost surely on $[0,T)$. Furthermore,  for any $t \in [0,T)$, $\limsup_{\beta \to \infty} m_{\beta}(t) \leq m(t)$. 
Thus, for any $t \in [0, T)$,we have that, 
\begin{align}
\int_0^t \psi(s) m_{\beta}(s) \de s \to \int_0^t \psi(s) m(s) \de s, \label{eqn:intermediate1}
\end{align}
 as $\beta \to \infty$, by the dominated convergence theorem.
Therefore, 
\begin{align}
&\Big| \int_0^{T} \psi(s) \de \nu_{\beta}(s) - \int_{0}^{T} \psi(s) \de \nu(s) \Big| \nonumber \\
&\leq \Big|  \int_0^{t} \psi(s) \de \nu_{\beta}(s) - \int_{0}^{t} \psi(s) \de \nu(s)\Big| + \Big|  \int_t^{T} \psi(s) \de \nu_{\beta}(s) - \int_{t}^{T} \psi(s) \de \nu(s) \Big|. \nonumber 
\end{align}
By \eqref{eqn:intermediate1},
\begin{align}
\limsup_{\beta \to \infty} \Big| \int_0^{T} \psi(s) \de \nu_{\beta}(s) - \int_{0}^{T} \psi(s) \de \nu(s) \Big| \leq \limsup_{\beta \to \infty} \Big|  \int_t^{T} \psi(s) \de \nu_{\beta}(s) - \int_{t}^{T} \psi(s) \de \nu(s) \Big|. \nonumber 
\end{align}
Finally, by triangle inequality, we have, 
\begin{align}
\Big|  \int_t^{T} \psi(s) \de \nu_{\beta}(s) - \int_{t}^{T} \psi(s) \de \nu(s) \Big| \leq \nu_{\beta}([0,T]) \max_{t \leq s \leq T} | \psi(s) - \psi(T)| + | \nu_{\beta}((t, T])- c | + \int_{t}^{T} m(s) \de s. \nonumber 
\end{align}
We let $\beta \to \infty$, and let $ t \uparrow T$  to conclude that 
\begin{align}
\Big|  \int_t^{T} \psi(s) \de \nu_{\beta}(s) - \int_{t}^{T} \psi(s) \de \nu(s) \Big|  \to 0 \nonumber 
\end{align}
as $\beta \to \infty$. This completes the proof. 

\end{proof}
\section{Constrained ground states on the discrete hypercube}
In this section, we derive a variational representation for the ground state of a mixed Ising spin glass, with an additional magnetization constraint. For $\sigma \in \{ \pm 1\}^{N}$, we set $m(\sigma) = \frac{1}{N} \sum_{i} \sigma_i$. 
\begin{thm}
\label{thm:constrained-GS-SK} For any $a \in[-1,1]\cap \mathbb{Q}$ and $\epsilon_{N}\to0$,
we have, for $\xi$ convex,
\[
\lim_{N\to\infty}\E\max_{\substack{\sigma\in\{\pm1\}^{N}\\
m(\sigma)\in[a-\epsilon_{N},a+\epsilon_{N}]
}
}\frac{H(\sigma)}{N}=\inf_{\nu,h}[\mathcal{P}^{2}(\nu,h)-a\cdot h],
\]
where $\mathcal{P}^2$ is given in \eqref{defn:P2}. 
\end{thm}

\noindent
The proof of this result follows in a few steps. For $a\in[-1,1]$, define
\begin{align*}
G_{N}(a;\eta)  &=\E\max_{\substack{\sigma\in\{\pm1\}^{N}\\
m \in [a- \eta, a +\eta]
}
}\frac{H(\sigma)}{N}, \\
E_{N}(h) & =\E\max_{\substack{\sigma\in\{\pm1\}^{N}}
}\left[\frac{H(\sigma)}{N}+hm(\sigma)\right].
\end{align*}
As shown in \cite{auffchen2017}, we have 
\begin{equation}
\lim_{N \to \infty} E_{N}(h)=E(h)=\inf_{\nu}\mathcal{P}^{2}(\nu,h).\label{eq:E(h)-parisi-formula}
\end{equation}
We also have the following result which is folklore in the spin glass literature.
We include a proof for the convenience of the reader. 
\begin{lem}
For any $\eta>0$ and $a\in[-1,1],$ we have 
\[
\lim_{\eta\to0}\lim_{N\to\infty}G_{N}(a;\eta)=G(a)
\]
is well defined. Furthermore, $G(a)$ is concave and continuous for all $a\in[-1,1]$. 
\end{lem}

\begin{proof}
The result follows by a modification of the Guerra-Toninelli argument
\cite{GuerraToninelliSK}. We first show existence. Fix $\eta>0,$ and consider
\[
G_{N+M}(a;\eta)=\E\max_{m\in [a- \eta, a+ \eta]}\frac{H_{N+M}(\sigma)}{N+M}.
\]
For $N,M$ sufficiently large, $\{\sigma \in \{ \pm 1\}^{N+M}:  m(\sigma) \in [a-\eta, a+ \eta] \}$ is non-empty,
so that this is well defined.

We decompose $\sigma \in \{\pm 1\}^{N + M}$ as $\sigma = (\rho, \epsilon)$, with $\rho \in \{\pm 1\}^{N}$, $\epsilon \in \{\pm 1\}^M$.
Consider now the interpolating Hamiltonian for $\sigma=(\rho,\epsilon)$,
defined by
\[
H_{N+M,t}(\rho,\epsilon)=\sqrt{t}H_{N+M}(\sigma)+\sqrt{1-t}(H_{N}(\rho)+H_{M}(\epsilon)),
\]
for $t\in [0,1]$, where we view the Hamiltonians at different $N$'s to be independent.
For $\sigma^{i}=(\rho^{i},\epsilon^{i})$ with $i=1,2$, let 
\[
R_{12}=\frac{1}{N+M}(\sigma^{1},\sigma^{2}),\quad R_{12}^{\rho}=\frac{1}{N}(\rho^{1},\rho^{2}),\quad R_{12}^{\epsilon}=\frac{1}{M}(\epsilon^{1},\epsilon^{2}).
\]
For any $\beta>0$ define
\[
\phi_{\beta}(t)=\frac{1}{\beta\left(N+M\right)}\E\log\sum_{m\in [a-\eta, a+\eta]}e^{\beta H_{N,t}(\sigma)}.
\]
This satisfies  
\[
\phi'_{\beta}(t)=\frac{1}{{N +M} }\E\left\langle \partial_{t}H_{N,t}\right\rangle =\frac{\beta}{N+M }\E\left\langle C_{11}-C_{12}\right\rangle 
\]
with 
\[
C_{12}=\E[\partial_{t}H_{t}(\sigma^{1})H_{t}(\sigma^{2})]={(N+M)} \xi(R_{12})-N\xi(R_{12}^{\rho})-M\xi(R_{12}^{\epsilon}),
\]
where the second equality follows using Gaussian integration by parts.
Note that $C_{11}=0$. Furthermore, since $R_{12}=\lambda R_{12}^{\rho}+(1-\lambda)R_{12}^{\epsilon}$
with $\lambda=\frac{N}{N+M}$, we have 
\[
\xi(R_{12}) \leq \lambda\xi(R_{12}^{\rho})+(1-\lambda)\xi(R_{12}^{\epsilon})
\]
by convexity of $\xi$. Thus $\phi'_{\beta}(t)\geq 0$, so that
$\phi_{\beta}(0){\leq}\phi_{\beta}(1).$
Sending $\beta\to\infty$, this yields
\[
\E\max_{m\in [a-\eta, a+ \eta]}\frac{H_{N+M}(\sigma)}{N+M}\geq\E\max_{m\in [a- \eta, a+ \eta]}\frac{H_{N}+H_{M}}{N+M}
\]
Now note that since $m(\sigma)=\lambda m(\rho)+(1-\lambda)m(\epsilon)$,
we have that 
\[
\{\sigma=(\rho,\epsilon):m(\rho)\in [a-\eta, a+ \eta],m(\epsilon)\in [a-\eta,a+\eta]\}\subseteq\{\sigma:m(\sigma) \in [a-\eta, a+\eta]\}
\]
Consequently, $G_{N+M}(a;\eta)\geq\lambda G_{N}(a;\eta)+(1-\lambda)G_{M}(a;\eta).$
Thus the sequence $NG_{N}(a;\eta)$ is super-additive, so that $G_N(a;\eta)$ has a limit.

To establish continuity of $G(\cdot; \eta)$, fix $a,b \in [-1,1]$, and $\eta >0$. For each $\sigma \in \{ m(\sigma) \in [a-\eta, a+ \eta]\}$, let $\pi(\sigma)$ denote a configuration in $\{m(\sigma) \in [b - \eta, b+ \eta]\}$ which minimizes the Hamming distance. Similarly, for each $\sigma \in \{m(\sigma) \in [b-\eta, b+ \eta]\}$, let $\pi'(\sigma)$ denote the configuration with magnetization in $[a-\eta, a+ \eta]$ which minimizes the Hamming distance. Recall that $\mathrm{Var}(H_N(\sigma) - H_N(\tau) )= 2N (\xi(1) - \xi(R_{12} ))$, where $N \cdot R_{12} = (\sigma, \tau)$. Thus 
\begin{align}
\frac{1}{N} \max\Big\{ \max_{m\in[a-\eta, a+ \eta]} \mathrm{Var}\Big( H_N(\sigma) - H_N(\pi(\sigma))  \Big) , \max_{m\in[b-\eta, b+ \eta]} \mathrm{Var}\Big( H_N(\sigma) - H_N(\pi'(\sigma))  \Big) \Big\}\leq C(\xi') |a-b|\nonumber 
\end{align}
uniformly in $N$, where $C(\xi')$ is a universal constant dependent on $\xi'$. 
We have, 
\begin{align}
|G_N(a;\eta) - G_N(b; \eta)| \leq  \frac{1}{N} \E \max_{m\in[a-\eta, a+ \eta]} | H_N(\sigma) - H_N(\pi(\sigma)) | + \frac{1}{N} \mathbb{E} \max_{m\in[b-\eta, b+ \eta]}  | H_N(\sigma) - H_N(\pi'(\sigma)) |. \nonumber
\end{align}
Note that for any collection of dependent, centered gaussians $\{ Z_i: 1\leq i \leq r\}$, $\mathbb{E}[\max |Z_i|] \leq 2 \mathbb{E}[\max Z_i] \leq 2 \sqrt{2 \sigma^2 \log r}$, where $\sigma^2 = \max_{1\leq i \leq r} \mathrm{Var}(Z_i)$. Applying this bound individually to the two terms in the display above, with $r \leq 2^N$, we get 
\begin{align}
|G_N(a;\eta) - G_N(b; \eta)| \leq C'(\xi') \sqrt{|a-b|} \nonumber 
\end{align} 
for some universal constant $C'(\xi')>0$, uniformly in $N$. 
Finally, we let $N \to \infty$, and obtain the continuity of $G(\cdot;\eta)$.

To obtain concavity, let $\lambda\in[0,1]$ be rational and work along a subsequence
such that $\lambda N$ is an integer. 
 By the same interpolation
argument, we obtain that 
\[
G_{N}(\lambda a+(1-\lambda)b;\eta)\geq\lambda G_{\lambda N}(a;\eta)+(1-\lambda)G_{(1-\lambda)N}(b;\eta).
\]
Sending $N\to\infty$ yields
\[
G(\lambda a+(1-\lambda)b; \eta)\geq\lambda G(a; \eta)+(1-\lambda)G(b; \eta).
\]
Using the continuity of $G(\cdot; \eta)$, we obtain the concavity of  $G(\cdot;\eta)$ for each $\eta$. 

Finally, since the map $\eta \mapsto G_N(a;\eta)$
is increasing, we see $\eta\mapsto G(a;\eta)$ is as well. Thus the pointwise limit $G(a)=\lim_{\eta\to0} G(a;\eta)$
is well-defined, continuous, and concave.
\end{proof}
\begin{proof}[Proof of Theorem \ref{thm:constrained-GS-SK} ]
Since $G(a)$ is continuous, we may apply a standard covering argument
(along with the definitions of $G_{N}$ and $E_{N}$) to obtain
\[
E(h)=\max_{a\in[-1,1]}\{ha+G(a)\}.
\]
Extending $G(a)$ by $-\infty$ off of $[-1,1]$, we have, by
concavity of $G(a)$, 
\[
G(a)=\inf_{h}[E(h)-ha].
\]
The result then follows by \eqref{eq:E(h)-parisi-formula}. 
\end{proof}

\bibliographystyle{plain}
\bibliography{graphref}

\end{document}